\documentclass{amsart}
\usepackage{amsfonts,amssymb,amsmath}

\newtheorem{theorem}{Theorem}
\newtheorem{lemma}[theorem]{Lemma}
\newtheorem{cor}[theorem]{Corollary}
\newtheorem{prop}[theorem]{Proposition}
\newtheorem{defn}[theorem]{Definition}

\newtheorem{exam}[theorem]{Example}

\newcommand{\R}{\mathbb{R}}
\newcommand{\N}{\mathbb{N}}
\newcommand{\Q}{\mathbb{Q}}

\newcommand{\K}{\mathcal{K}}

\newcommand{\Filt}{\mathcal{F}}

\newcommand{\vx}{\mathbf{x}}
\newcommand{\vy}{\mathbf{y}}
\newcommand{\vz}{\mathbf{z}}

\newcommand{\cl}[1]{\overline{#1}}

\newcommand{\tq}{\ge_T}
\newcommand{\te}{=_T}

\newcommand{\Fan}{F_\omega}

\newcommand{\zero}[1]{\mathop{zero} (#1)}
\newcommand{\coz}[1]{\mathop{coz} (#1)}

\newcommand{\Zero}[1]{\mathop{Z} (#1)}
\newcommand{\Coz}[1]{\mathop{Coz} (#1)}

\newcommand{\down}[1]{\downarrow\! #1}

\begin{document}
\title{Point Networks for Special Subspaces of $\R^{\kappa}$}

\author{Ziqin Feng}
\address{Department of Mathematics and Statistics, Auburn University, AL~36849, USA}
\email{zzf0006@auburn.edu}

\author{Paul Gartside}
\address{Department
    of Mathematics, University of Pittsburgh, Pittsburgh, PA~15260, USA}
\email{gartside@math.pitt.edu}

\begin{abstract}
Uniform characterizations of certain special subspaces of products of lines are presented. The characterizations all involve a collection of subsets (base, almost subbase, network or point network) organized by a directed set. New characterizations of Eberlein, Talagrand and Gulko compacta follow.
\end{abstract}



\maketitle

\section*{Introduction} The purpose of this paper is to give uniform characterizations of certain special subspaces of products of lines that arise in analysis. The characterizations all involve two objects: first an order, and second a collection of subsets of the space `organized' by the order. For the collection of subsets we take bases, almost subbases, networks and, most importantly,  point networks. As a result we derive some new (and re-derive some old) characterizations of Eberlein, Gulko and Talagrand compacta. The characterizations in terms of point networks  yield new clean proofs that the continuous image of a compact space which is Eberlein, Gulko or Talagrand has the same respective property. Rudin's original proof \cite{BRW77} that the continuous image of an Eberlein compact space is Eberlein compact is widely acknowledged to be involved.

\subsubsection*{Partial Order Preliminaries} All our partially ordered sets, $P$, will be directed (given $p,q$ in $P$ there is an $r$ in $P$ such that $r \ge p$ and $r \ge q$). It turns out that what is important about our directed sets happens cofinally. So we compare directed sets $P$ and $Q$ via the Tukey order. A map $\phi : P \to Q$ is a Tukey quotient if $\phi$ is order-preserving and $\phi(P)$ is cofinal in $Q$.  Write $P \tq Q$ if there is a Tukey quotient of $P$ to $Q$. We say `$P$ and $Q$ are Tukey equivalent', and write $P \te Q$, if $P \tq Q$ and $Q \tq P$.
We note that $P \tq P \times \N$ unless $P$ is countably directed (every countable subset of $P$ has an upper bound).

The results in the first section of this paper apply to \emph{all} directed sets $P$. In the second section we improve the characterizations obtained in the first part by using specific properties of the partial orders associated with Eberlein, Talagrand and Gulko compacta --- respectively: the natural numbers, $\N$, the product order on $\N^\N$, and the set of compact subsets, $\K(M)$, ordered by inclusion, of a separable metrizable space $M$. It is now well-known that $\K(\N^\N) \te \N^\N \tq \N \tq 1$; for compact $M$, $\K(M) \te 1$; for locally compact separable metrizable $M$, $\K(M) \tq \N$; and for any non-locally compact separable metrizable $M$, we have $\K(M) \tq \K(\N^\N)$.

Given a directed set $P$, a collection $\mathcal{C}$ of subsets of a space $X$ is said to be  `$P$-ordered'  if  we can write $\mathcal{C}=\bigcup \{ \mathcal{C}_p : p \in P\}$ such that $p \le p'$ implies $\mathcal{C}_p \subseteq \mathcal{C}_{p'}$. Observe that if $\phi$ is a Tukey quotient of $Q$ to $P$, then  $\mathcal{C}$ is also a $Q$-ordered cover, indeed we can write $\mathcal{C}=\bigcup \{\mathcal{C}'_q : q \in Q\}$ where $\mathcal{C}'_q=\mathcal{C}_{\phi(q)}$. Further, if $\mathcal{Q}$ is a property that a collection of subsets of $X$ might have, then we say that the collection $\mathcal{C}$ is $P$-$\mathcal{Q}$ if it can be $P$-ordered as $\mathcal{C}=\bigcup \{\mathcal{C}_p : p \in P\}$ such that every subcollection $\mathcal{C}_p$ has property $\mathcal{Q}$.

\subsubsection*{Topological Preliminaries} Our topological notation is standard (see \cite{Eng} for example) except we introduce the following uniform method for dealing with compactness and the Lindelof property. Let $\kappa$ be a cardinal.
We say that a space $Y$ is `$<\!\!\kappa$-compact' if every open cover of $Y$ has a subcover of size $<\!\!\kappa$. Then `$<\!\!\aleph_0$-compact' is standard compactness and `$<\!\!\aleph_1$-compact' is the Lindelof property. Other less well known terms (such as `almost subbase' and `point network') will be defined below as it becomes appropriate. All spaces are assumed to be Tychonoff.

\section*{Almost Subbases, Bases, Networks and Point Networks}

\subsection*{Almost Subbases}

Almost subbases were introduced by Dimov \cite{Dim} where, among other things, he used them to characterize the subspaces of Eberlein compacta. We follow his notation and terminology.

Let $X$ be a space. Take any function $f:X \to [0,1]$. The zero set of $f$, denoted $\zero{f}$, is $f^{-1}\{0\}$, while the cozero set of $f$, written $\coz{f}$, is the complement in $X$ of $\zero{f}$. A subset $S$ of $X$ is a (co)zero set if it is the (co)zero set of some continuous function. Write $\Zero{X}$ for the collection of zero subsets of $X$, and $\Coz{X}$ for the collection of cozero subsets. Let $V$ be a subset of $X$. If there exists a collection $U(V)=\{U_n(V): n\in \N\}$ such that $V=\bigcup \{U_n(V):n\in \N\}$ and $U_n(V)\subseteq U_{n+1}(V)$, $U_{2n-1}(V)\in \Zero{X}$, and $U_{2n}\in \Coz{X}$ for each $n\in \N$, then we say that $V$ is $U$-representable and the collection $U(V)$ is a $U$-representation of $V$. We record a simple but useful fact about $U$-representations.
\begin{lemma}\label{fV}\cite{Dim} Let $X$ be a space and let $U(V)=\{U_n(V): n\in \N\}$ be a $U$-representation of a subset $V$ of $X$. Then there exists a continuous function $f_V: X\rightarrow [0,1]$ such that $V=\coz{f_V}$ and $f_V^{-1}[1/(2n-1), 1]=U_{2n-1}(V)$, for every $n\in \N$.
\end{lemma}
Henceforth, given a $U$-representation for a set $V$, we consider such an $f_V$, satisfying the above lemma, as being fixed.

If $\alpha$ is a family of subsets of $X$ and, for each $V\in \alpha$, $U(V)$ is a $U$-representation of $V$, then the family $U(\alpha)=\{U(V):V\in \alpha\}$ is called a $U$-representation of $\alpha$.
\begin{defn} A family $\alpha$ of subsets of a space $X$ is said to be an almost subbase of $X$ if there exists a $U$-representation $U(\alpha)$ of $\alpha$ such that the family $\alpha\cup \{X\setminus U_{2n-1}(V): V\in \alpha, n\in \N\}$ is a subbase of $X$.
\end{defn}

A family $\alpha$ of subsets of a space $X$ is $F$-separating if whenever $x$ and $y$ are distinct points in $X$ then there is a $V$ in $\alpha$ such that $x$ is in $V$ but $y$ is not in $\cl{V}$ or vice versa ($y \in V$, but $x \notin \cl{V}$). A given almost subbase need not be $F$-separating. The next lemma allows us to find one which is --- and to do so in a manner which respects order.

\begin{lemma}\label{f_sep} Let $X$ be a space, and $P$ a directed set. Let $\alpha$ be a $P$-point-$<\!\!\kappa$ almost subbase.
Define $\alpha'= \{f_V^{-1}(r,1]: V\in \alpha, r\in \Q \cap [0,1]\}$.
Then $\alpha'$ is $(P\times\N)$-point-$<\!\!\kappa$ almost subbase which is $F$-separating.
\end{lemma}
\begin{proof} Note that $\alpha = \{ f_V^{-1} (0,1] : V \in \alpha\}$, so $\alpha$ is a subset of $\alpha'$. Hence $\alpha'$ is an almost subbase. It is straight forward to check (and observed in \cite{Dim2}) that $\alpha'$ is $F$-separating.

We know $\alpha$ can be $P$-ordered, say as $\alpha = \bigcup \{\alpha_p : p \in P\}$, where every $\alpha_p$ is point-$<\!\!\kappa$. Fix an enumeration $\{q_n\}_{n \in \N}$ of $[0,1] \cap \Q$. For $p$ in $P$ and $n$ in $\N$, set $\alpha'_{p,n} = \{ f_V^{-1} (q_i,1] : i=1, \ldots , n, \ V \in \alpha_p\}$. It is clear that $\alpha' = \bigcup \{ \alpha'_{p,n} : (p,n) \in P \times \N\}$, and $\alpha'_{p,n} \subseteq \alpha'_{p',n'}$ when $(p,n) \le (p',n')$. So we have a $(P\times \N)$-ordering of $\alpha'$. Further, for a fixed $p$, we know that $\alpha'_p$ is point-$<\!\!\kappa$. So for a fixed $i$ the collection $\{f_V (q_i,1] : V \in \alpha_p\}$ is also point-$<\!\!\kappa$. Hence, for any $p$ in $P$ and $n$ in $\N$, $\alpha'_{p,n}$, as a finite union of point-$<\!\!\kappa$ collections, is point-$<\!\!\kappa$.
\end{proof}

The next lemma is clear.
\begin{lemma}\label{as_hered}
Let $X$ be a space, $A$ a subspace of $X$, and $P$ a directed set. Let $\alpha$ be a $P$-point-$<\!\!\kappa$ almost subbase.
Define $\alpha_A= \{V_A : V_A=V \cap A, V \in \alpha\}$, and $U(V_A) = \{ U_n(V) \cap A : n \in \N\}$.
Then $\alpha_A$ is $P$-point-$<\!\!\kappa$ almost subbase with respect to the stated $U$-representation for $A$.
\end{lemma}

\subsection*{Almost Subbases characterize Embeddings}

For any free filter $\Filt$ on an infinite set $X$, write $X(\Filt)$ for the space with underlying set $X \cup \{\ast\}$, and topology where the points of $X$ are isolated and neighborhoods of $\ast$ are $\{\ast\} \cup U$ for $U$ in $\Filt$. We allow the possibility that $\Filt=\mathbb{P}(X)$, in which case $X(\Filt)$ has the discrete topology. Write  $A(\kappa)$, for the case $X=\kappa$, an infinite cardinal,  and $\Filt=\{(\kappa\setminus F): F\text{ is a finite subset of }\kappa\}$. It is the one point compactification of a discrete space of size $\kappa$. Write $L(\kappa)$, for the case $X=\kappa$  and $\Filt=\{(\omega_1\setminus C): C \text{ is a countable subset of }\kappa\}$. It is the one point Lindelofication of a discrete space of size $\kappa$.

Recall that for any space $X$, $C_p(X)$ is the subspace of $\R^X$ consisting of all continuous $f: X \to \R$. When $X=X(\Filt)$, write $C_p(X,\ast)$ for all continuous $f:X \to \R$ such that $f(\ast)=0$. It is easy to check that $C_p(X(\Filt))$ and $C_p(X(\Filt),\ast)$ are homeomorphic.

Then it is clear that the standard $\Sigma$-product of $\kappa$ many lines, $\Sigma \R^\kappa$ is homeomorphic to $C_p(L(\kappa))$, while the standard $\Sigma_*$ product of $\kappa$ many lines, is homeomorphic to  $C_p(A(\kappa))$.

By definition a compact space is Corson compact if and only if it embeds in some $\Sigma$-product, and so if and only if it embeds in some $C_p(L(\kappa))$. It is well known that a compact space is Eberlein compact if and only if it embeds in some $\Sigma_*$ product, and so if and only if it embeds in some $C_p(A(\kappa))$.

Mercourakis \cite{Merc} and Sokolov \cite{Sokol} have shown that a compact space is Talagrand (respectively, Gulko) compact if and only if it embeds in some $C_p(X(\Filt))$ where $X(\Filt)$ has an $\N^\N$-ordered cover by compact sets (respectively, where $X(\Filt)$ has a $\K(M)$-ordered compact cover, for some separable metrizable space $M$).

\begin{theorem}\label{as_is_embed} Let $P$ be a directed set and $\kappa$ an infinite cardinal.

(1) Every space $Y$ which embeds in some $C_p(X(\Filt))$ where $X(\Filt)$ has a $P$-ordered cover by $<\!\!\kappa$-compact sets has a $(P\times\N)$-point-$<\!\!\kappa$ almost subbase.

(2) Every space $Y$ with a $P$-point-$<\!\!\kappa$ almost subbase embeds in some $C_p(X(\Filt))$ where $X(\Filt)$ has a $(P\times \N)$-ordered cover by $<\!\!\kappa$-compact sets.

If $P$ is not countably directed then `$P \times \N$' can be replaced simply by `$P$' in both (1) and (2).
\end{theorem}

\begin{proof} For claim (1) suppose  $\mathcal{C}=\{C_p: p\in P\}$ is a $P$-ordered cover of $X(\Filt)$ by $<\!\!\kappa$-compact sets.
By Lemma~\ref{as_hered} it will suffice to show that $C_p(X(\Filt), \ast)$ has a $(P\times\N)$-point-$<\!\!\kappa$ almost subbase.

Let $\mathcal{B}=\{B_n: n\in \N\}$ be a countable base of  $\R\setminus \{0\}$ consisting of open intervals such that $B_n\subseteq \R\setminus [-1/n, 1/n]$.
For points $x_1, \ldots, x_m$ in $X(\Filt)$ and open intervals $S_1, \ldots , S_m$ of $\R$ let $O(x_1, \ldots, x_m; S_1, \ldots, S_m)= \{f \in C_p(X(\Filt), \ast) : f(x_i)\in S_i , \text{ for } i=1, \ldots , m\}$.  Then $O(x_1, \ldots, x_m; S_1, \ldots, S_m)$ is a cozero  subset of $C_p(X(p), \ast)$.

For any $p\in P$ and $n \in \N$, let $\mathcal{O}_{p,n}=\{O(x_1, \ldots, x_m; S_1, \ldots, S_m): x_i\in C_p \text{ and each } S_i$ is either one of $B_1, \ldots , B_n$ or $(-\infty,-1/j)$ or $(1/j,\infty)$ for $j=1, \ldots, n\}$.
Note that if $(p,n)$ and $(p',n')$ are in $P \times \N$ and $(p,n) \le (p',n')$ (in other words $p \le p'$ and $n \le n'$) then $\mathcal{O}_{p,n} \subseteq \mathcal{O}_{p',n'}$.
We will verify that every family $\mathcal{O}_{p,n}$ is point-$<\!\!\kappa$, and hence $\mathcal{O}=\bigcup\{\mathcal{O}_{p,n} : p\in P \text{ and } n\in \N\}$ is $(P\times\N)$-point-$<\!\!\kappa$.

 Suppose, for contradiction, $\mathcal{O}_{p,n}$ is not point-$<\!\!\kappa$ for some $p\in P$ and $n\in \N$, in other words there is an $f\in C_p(X(\Filt),\ast)$ and $\{O_{\beta}:\beta<\kappa \}\subseteq \mathcal{O}_{p,n}$ such that $f\in O_\beta$ for each $\beta<\kappa$. Since $\kappa\geq \aleph_0$, there must be a $\kappa$ sized subset $\{x_\gamma: \gamma<\kappa\}\subseteq C_p$ such that $f(x_\gamma)\in B$ for some $B$ from either $B_1,\ldots, B_n$, or $(-\infty,-1/j)$ or $(1/j,\infty)$ with $j=1, \ldots, n$. Hence $|f(x_\gamma)|>1/n$ for each $\gamma<\kappa$ which contradicts the fact that $C_p$ is $<\!\!\kappa$-compact.

We now show that $\mathcal{O}$ is an almost subbase of $C_p(X(\Filt), \ast)$. For any open interval $(a,b)$ in  $\R$, we define a $U$-representation of the interval by $U((a, b))=\{U_n((a, b)): U_{2n}((a,b))=(a+1/2n, b-1/2n)$ and $ U_{2n-1}((a,b))=[a+1/(2n-1), b-1/(2n-1)]\}$. Then define the $U$-representation of $\mathcal{O}$ as $U(O(x_1, \ldots, x_m; S_1, \ldots, S_m))=\{O(x_1,\ldots, x_m; U_n(S_1), \ldots, U_n(S_m)): n\in \N$ with $x_i\in U_n(S_i)$ and $S_i\in \mathcal{B}\cup \{(-\infty,-1/j), (1/j, +\infty): j\in \N\}$ for $i=1, \ldots, m\} $.

Next we establish that $\mathcal{O}\cup\{C_p(X(\Filt),\ast) \setminus U_{2n-1}(V): V\in \mathcal{O}, n\in \N\}$ is a subbase of $C_p(X(\Filt),\ast)$. We take $f\in C_p(X(\Filt),\ast)$ and a basic open neighborhood of $f$, namely, $O(x_1, \ldots, x_k: I_1, \ldots, I_k)$ with $I_1, \ldots, I_k$ being open intervals. Then there exists $p\in P$ such that $x_1,\ldots, x_k\in C_p$. Relist the sets of points $\{x_1, \ldots, x_k\}$ as $\{x_1, \ldots, x_{\ell_1}\}\cup \{y_1, \ldots, y_{\ell_2}\}$ such that $f(x_i)\ne 0$ for $i=1, \ldots, \ell_1$ and $f(y_i)=0$ for $i=1, \ldots, \ell_2$. Relist the intervals $I_1, \ldots , I_k$    as $I_1, \ldots, I_{\ell_1}$ and $\hat{I}_1, \ldots, \hat{I}_{\ell_2}$ with $x_i\in I_{i}$ and $y_j\in \hat{I}_{\ell}$ for $i=1, \ldots, \ell_1$ and $j=1, \ldots, \ell_2$.

We will show we can choose $O$, $O_1, \ldots, O_{2\ell_2}\in \mathcal{O}$ and $N\in \N$ such that $f\in O\cap \bigcap \{C_p(X(\Filt),\ast)\setminus U_{2N-1}(O_{i}): i=1, \ldots,2\ell_2\}\subseteq O(x_1, \ldots, x_k; I_1, \ldots, I_k)$. Choose $N$ big enough such that: (i) there exist $B_{i_1},\ldots, B_{i_{\ell_1}}\in \{B_1, \ldots, B_N\}$ with $x_j\in B_{i_j}\subseteq I_j$ for $j=1, \ldots, \ell_1$;   and (ii) $\R\setminus ((-\infty, -1/N-1/(2N-1)]\cup [1/N+1/(2N-1),\infty))\subseteq \hat{I}_j$ for $j=1, \ldots, \ell_2$. Let $O=O(x_1, \ldots, x_{\ell_1}; B_{i_1}, \ldots, B_{i_{\ell_1}})$, $O_j=O(y_j; (1/N, \infty))$ for $j=1, \ldots, \ell_2$ and $O_j=O(y_j; (-\infty, -1/N))$ for $j=\ell_2+1, \ldots, 2\ell_2$. Then we see that $f\in O\cap \bigcap \{C_p(X(\Filt))\setminus U_{2N-1}(O_{j}): j=1, \ldots,2\ell_2\}\subseteq O(x_1, \ldots, x_k; I_1, \ldots, I_k)$.

\medskip

For claim (2) suppose $Y$ is a space with a $P$-point-$<\!\!\kappa$ almost subbase. Let $Q=P \times \N$. By Lemma~\ref{f_sep} there is an $F$-separating  $Q$-point-$<\!\!\kappa$ almost subbase for $Y$, say $\alpha=\bigcup\{\alpha_q: q \in Q\}$. For each $V\in \alpha$, let $U(V)$ be the corresponding $U$-representation of $V$. We show $Y$ can be embedded in a $C_p(X(\Filt))$ where $X(\Filt)$ has a $(Q\times \N)$-ordered compact cover. Since $\N \te \N \times \N$, we see that $Q \times \N \te P \times \N$. So $X(\Filt)$ has a $(P\times \N)$-ordered compact cover, as desired.

For any $n\in \N$ and $V\in \alpha$ there exists a continuous function $a_n^V: Y\rightarrow \R$ with $a_n^V(y)=1$ for $y\in U_{2n-1}(V)$ and $a_n^V(y)=0$ for $y\in Y\setminus U_{2n}(V)$. By Lemma \ref{fV}, we can assume that $a_n(x)\neq 1$ for all $x\notin U_{2n-1}(V)$. Let $X=\{a_n^V: n\in \N \text{ and } V\in \alpha\}$. Let $C_{q,n}=\{a_i^V: i\le n \text{ and }V\in \alpha_q\}$ for $n\in \N$ and $q\in Q$. Then we can see that, if $(q,n) \le (q',n')$ then $C_{q,n}\subseteq C_{q',n'}$. Then we define a collection $\Filt$ of subsets of $X$ as $F\in\Filt$ if and only if $C_{q,n}\setminus F$ has cardinality $<\!\!\kappa$ for each $q\in Q$ and $n\in \N$. For each $q\in Q$ and $n\in \N$, we can see that $C_{q,n}$ is a $<\!\!\kappa$-compact subset of $X(\Filt)$. Also, we can see that: (i) if $F_1$ and $F_2$ are in $\Filt$, $C_{q,n}\setminus (F_1\cap F_2)=(C_{q,n}\setminus F_1)\cup (C_{q,n}\setminus F_2)$ also has cardinality $<\!\!\kappa$ for each $q$ and $n$, in other words, $F_1\cap F_2$ is in $\Filt$; and (ii) if $F_1\in \Filt$ and $F_1\subseteq F_2$, then $C_{q,n}\setminus F_2$ has size $<\!\!\kappa$ for each $q, n$, so  $F_2$ is in $\Filt$. The space $X(\Filt)$ evidently has a $Q\times\N$-ordered cover, $\{C_{q,n}: q\in Q, n\in \N\}$, of $<\!\!\kappa$-compact sets. It remains to show that $Y$ can be embedded in $C_p(X(\Filt))$.


We define a mapping $H: Y\rightarrow C_p(X(\Filt))$ by $H(y)(a_{n}^{V})=a_{n}^{V}(y)$ for each $y\in Y$. Since $\alpha_q$ is point-$<\!\!\kappa$ for each $q\in Q$, $H(y)(a_n^V)\neq 0$ only at $<\!\!\kappa$ many $a_n^V$ for each $q\in Q$ and $n\in \N$. So $H(y)$ is continuous on the space $X(\Filt)$ for each $y\in Y$. Hence $H$ is a well-defined map from $Y$ into $C_p(X(\Filt))$.

We will show that the mapping $H$ is injective, continuous, and open onto its image. We take distinct $y_1, y_2\in Y$. Since $\alpha$ is $F$-separating (interchanging $y_1$ and $y_2$ if necessary) there exists a $V\in \alpha$ such that $y_1\in V$ and $y_2\notin \cl{V}$. Then there exists $N\in \N$ such that $x\in U_{2N-1}(V)$, and $H(y_1)\neq H(y_2)$ because $H(y_1)(a_{N}^{V})=1$ while $H(y_2)(a_{N}^{V})=0$. Thus the mapping $H$ is injective.

Choose a set of open intervals $\{I_1, \ldots, I_k\}$ in the real line $\R$. The set $O=O(a_{n_1}^{V_1}, \ldots, a_{n_k}^{V_k}; I_1, \ldots, I_k)$ is a basic open neighborhood of $C_p(X(\Filt))$. Then $H^{-1}(O)=\bigcap\{(a_{n_i}^{V_i})^{-1}(I_i): i=1, \dots, k\}$. By the continuity of $a_{n_i}^{V_i}$ for $i=1, \ldots, k$, $H^{-1}(O)$ is an open subset of $Y$. Therefore, the mapping $H$ is continuous.

Since the mapping $H$ is injective and $\alpha$ is an almost subbase, to show $H$ is onto its image it is enough to show that $H(V)$ and $H(Y\setminus U_{2n-1}(V))$ are open for every $V\in \alpha$ and $n\in \N$. By the definition of $H$, we can see that, for each $n\in\N, V\in \alpha$ and for each $y\in Y$,
$H(y)(a_n^V)>0\Leftrightarrow a_n^V(y)>0\ \text{and hence} \ y\in U_{2n}(V)\subset V\quad (\star)$, while $H(y)(a_n^V)<1\Leftrightarrow a_n^V(y)<1\ \text{and thus} \  y\notin U_{2n-1}(V)\quad (\ast)$.

Take any $\vy\in H(V)$. Let $y=H^{-1}(\vy)$. Choose $N$ such that $y\in U_{2N-1}(V)$, then $a_{N}(y)=1$. Let $W=O(a_{N}^{V};(1/2, +\infty))\cap H(Y)$ which is open in $H(Y)$. For each $\vz\in W$, let $z=H^{-1}(\vz)$. We see that $H(z)(a_N^V)>1/2$, so $z\in V$ by $(\star)$. Hence $H(z)\in H(V)$. Therefore, we have $\vy\in W\subset H(V)$. So  $H(V)$ is an open set in $H(Y)$.

Finally, take any $\vy\in H(Y\setminus U_{2n-1}(V))$. Let $y=H^{-1}(\vy)$, and note $a_{n}^{V}(y)<1$. Let $W=O(a_{n-1}^V;(-\infty, 1))\cap H(Y)$ which is also open in $H(Y)$. For each $\vz\in W$, let $z=H^{-1}(\vz)$. We see that $H(z)(a_{n}^V)<1$, so $z\notin U_{2n-1}(V)$ by $(\ast)$. Hence $H(z)\in H(Y\setminus U_{2n-1}(V))$. Therefore,  $\vy\in W\subset H(X\setminus U_{2n-1}(V))$. So $H(X\setminus U_{2n-1}(V))$ is also open in $H(Y)$, and the proof is complete.
\end{proof}

\subsection*{Almost Subbases give Bases}

Let $\mathcal{C}$ be a family of sets. For any $x$, write $(\mathcal{C})_x=\{C\in\mathcal{C}: x\in C\}$. The collection $\mathcal{C}$ is said to be $<\!\!\kappa$-Noetherian if every subcollection of $\mathcal{C}$ which is well-ordered by $\subseteq$ has the cardinality $<\!\!\kappa$. The family is said to be additively $<\!\!\kappa$-Noetherian if the collection of all unions of members of the family is $<\!\!\kappa$-Noetherian. We extend the notation to say that $\mathcal{C}$ is $\kappa$-Noetherian (resp. additively $\kappa$-Noetherian) if $\mathcal{C}$ is $<\!\!\kappa^+$-Noetherian (resp. additively $<\!\!\kappa^+$-Noetherian). Instead of (additively) $<\!\!\aleph_0$-Noetherian we simply say (additively) Noetherian.

\begin{lemma}\label{Noeth} Let $\mathcal{C}$ be a family of sets. The following are equivalent:

(1) $\mathcal{C}$ is additively $<\!\!\kappa$-Noetherian,

(2) the collection of all $<\!\!\kappa$ unions of members of $\mathcal{C}$ is $<\!\!\kappa$-Noetherian,

(3) every $\geq \kappa$ subcollection, say $\mathcal{C}_2$, of $\mathcal{C}$ contains a $\kappa$ sized subcollection $\mathcal{C}_1$ with a $<\!\!\kappa$ sized subcollection $\mathcal{C}_0$ such that $\bigcup \mathcal{C}_0 \supseteq \bigcup \mathcal{C}_1$, and

(4) it is not possible to find for all $\alpha < \kappa$, sets $C_\alpha$ from $\mathcal{C}$ and points $x_\alpha$, such that for all $\beta < \alpha < \kappa$ the point $x_\alpha$ is in $C_\alpha$ but not in $C_\beta$.
\end{lemma}

\begin{proof}
It is easy to verify that (1) implies (2).

We show (2) implies (3). Take any  subcollection $\mathcal{C}_2$ of $\mathcal{C}$ of size at least $\kappa$. Take any (exactly) $\kappa$-sized subcollection  $\mathcal{C}_1=\{C_\alpha : \alpha < \kappa\}$ of $\mathcal{C}_2$. For each $\alpha$ in $\kappa$, let $\hat{C}_\alpha = \bigcup \{ C_\beta : \beta \le \alpha\}$. Then each $\hat{C}_\alpha$ is a $<\!\!\kappa$ union of elements of $\mathcal{C}$, and if $\alpha' < \alpha$ then $\hat{C}_{\alpha'} \subseteq \hat{C}_{\alpha}$. So by (2), there must be an $\alpha_0$ such that for all $\alpha \ge \alpha_0$ we have $\hat{C}_{\alpha_0} = \hat{C}_{\alpha}$. Thus $\mathcal{C}_0=\{ C_\beta : \beta \le \alpha_0\}$ is a $<\!\!\kappa$ subcollection of $\mathcal{C}_1$ such that  $\bigcup \mathcal{C}_0 \supseteq \bigcup \mathcal{C}_1$, as required for (3).

Next we show that $\neg$(4) implies $\neg$(3). We know that there is a family of elements of $\mathcal{C}$ and points, $\{ (C_\alpha, x_\alpha) : \alpha \in \kappa\}$ as in the statement of (4). Let $\mathcal{C}_2=\{C_\alpha : \alpha \in \kappa\}$. Take any $\kappa$ sized subcollection $\mathcal{C}_1$ of $\mathcal{C}_2$. Write $\mathcal{C}_1=\{C_\alpha : \alpha \in S_1\}$ where $S_1$ is some $\kappa$ sized subset of $\kappa$. For any $\kappa$ sized subset $S_0$ of $S_1$, pick an $\alpha_0$ in $S_1$ such that $\alpha_0 > \max S_0$. Then $x_{\alpha_0} \notin \bigcup \{ C_\beta : \beta \in S_0\}$. Thus the negation of (3) holds.

Finally we show $\neg$(1) implies $\neg$(4). Suppose we have a $\kappa$ sized subfamily $\mathcal{C}_1$ of $\mathcal{C}$ which is well-ordered by (strict) inclusion. By transfinite induction we can easily find, for each $\alpha < \kappa$, a $C_\alpha$ in $\mathcal{C}_1$ such that $D_\alpha=C_\alpha \setminus \bigcup \{ C_\beta : \beta < \alpha\} \ne \emptyset$. Picking $x_\alpha \in D_\alpha$, for each $\alpha$ in $\kappa$, yields the sets $C_\alpha$ and points $x_\alpha$ required to negate (4).
\end{proof}

\begin{theorem}\label{as_b} Every productively $<\!\!\kappa$-compact space with a $P$-point-$<\!\!\kappa$ almost subbase has a $(P\times \N)$-point additively $<\!\!\kappa$-Noetherian base.
\end{theorem}

\begin{proof}
Let $X$ be a productively $<\!\!\kappa$-compact space which has a $P$-point-$<\!\!\kappa$ almost subbase $\alpha$. So $\alpha=\bigcup\{\alpha_p: p\in P\}$ where each $\alpha_p$ is point-$<\!\!\kappa$, $\alpha_p\subseteq \alpha_{p'}$ if $p\leq p'$, and the collection $\mathcal{S}=\alpha \cup \{ X \setminus U_{2n-1}(V) : V \in \alpha, \ n \in \N\}$ is a subbase for $X$. Let $\mathcal{B}$ be the base for $X$ consisting of all finite non-empty intersections of members of $\mathcal{S}$.   Then naturally we write  $\mathcal{B}=\bigcup\{\mathcal{B}_{p,m}: m\in \N \text{ and }p\in P\}$ where $B$ is is $\mathcal{B}_{p,m}$ if and only if $B=\bigcap_{i=1}^{k_1} V_i'\cap \bigcap_{j=1}^{k_2} \left(X\setminus U_{2n_j-1}(V_j)\right)$
 where $k_1, k_2 \in \N$, each
$k_1, k_2, n_j\leq m$, and each $V_i'$ and  $V_j$ is in $\alpha_p$. Clearly  $\mathcal{B}$ is $(P\times \N)$-ordered.

We show that $\mathcal{B}_{p,m}$ is point additively $<\!\!\kappa$-Noetherian.
To this end, fix $x$ in $X$ and any $\ge \kappa$ sized subset $\mathcal{B}'$ of $(\mathcal{B}_{p,m})_x$. Take any subcollection $\mathcal{B}_1$ with exactly $\kappa$ many elements. Since $\alpha_p$ is point-$<\!\!\kappa$,  we can suppose there is a fixed finite subset $\mathcal{\alpha}'$ of $(\alpha_p)_x$ and a fixed $k\in\N$ such that every $B$ in $\mathcal{B}$ has the form
\[ B = \bigcap \alpha' \cap \bigcap_{i=1}^k (X \setminus A(B,i)),\]
where each $A(B,i)$ is a closed member of $U(\alpha_p)$.

Since $X$ is productively $<\!\!\kappa$-compact, the space $X^k$ is $<\!\!\kappa$-compact. We will show there is a $<\!\!\kappa$ sized collection $\mathcal{B}_0$ contained in $\mathcal{B}_1$ such that $\bigcup \mathcal{B}_0 \supseteq \bigcup \mathcal{B}_1$, thereby verifying that $\mathcal{B}_{p,m}$ satisfies claim~(3) of the preceeding lemma.

By  definition of $\mathcal{B}_{p,m}$, the collection $\{A(B, i): B\in \mathcal{B}_1\}$ is point-$<\!\!\kappa$ for each $i=1, 2, \ldots, k$. Then the collection $\mathcal{C}=\{\prod_{i=1}^k A(B, i): B\in\mathcal{B}_1\}$ is a point-$<\!\!\kappa$ family of closed subsets in $X^k$.  Since $\mathcal{B}_1$ has cardinality $\kappa$, the collection $\mathcal{C}$ also has cardinality $\kappa$, hence it must have empty intersection. By $<\!\!\kappa$-compactness of $X^k$, $\mathcal{C}$ can not have the $\kappa$ Intersection Property. So there is a $<\!\!\kappa$ sized subcollection $\mathcal{C}_0$ of $\mathcal{C}$ such that $\bigcap \mathcal{C}_0 = \emptyset$. For each $C\in \mathcal{C}_0$, choose $B_C$ from $\mathcal{B}_1$ such that $C=\prod_{i=1}^k A(B_C, i)$. Then let $\mathcal{B}_0=\{B_C: C\in \mathcal{C}_0\}$. Hence $\mathcal{B}_0$ has cardinality $<\!\!\kappa$.

Take any $y$ in some element of $\mathcal{B}_1$, say $B'= \bigcap \alpha' \cap \bigcap_{i=1}^k (X \setminus A(B',i))$. Then the vector $(x_1, x_2, \ldots, x_k)$ with $x_i=y$ for each $i=1,2, \ldots, k$ is not in $C^{\star}$ for some $C^{\star}\in\mathcal{C}_0$. Hence $y\notin A(B_{C^\star}, i)$ for each $i=1,2,\ldots, k$,  and so we have $y\in B_{C^\star}$. Thus $\bigcup \mathcal{B}_0 \supseteq \bigcup \mathcal{B}_1$, as required.
\end{proof}

\subsection*{Expandable Networks and Point Networks}

Let $X$ be a space. An expandable network for $X$ is a collection $\mathcal{N}$ of pairs of subsets $(N,V)$ of $X$, where $N \subseteq V$ and $V$ is open such that if a point $x$ of $X$ is in an open set $U$ then there is an element $(N,V)$ of $\mathcal{N}$ such that $x \in N \subseteq U$. Clearly if $\mathcal{N}$ is an expandable network for $X$, then $\mathcal{N}' = \{ N : (N,V) \in \mathcal{N}\}$ is a (standard) network for $X$, and if $\mathcal{N}'$ is a network, then $\mathcal{N} = \{(N,X) : N \in \mathcal{N}'\}$ is an expandable network. The interest in expandable networks (as with almost subbases) comes when they are structured as a $P$-$\mathcal{Q}$ family.

Extending our previous notation and definition, let $\mathcal{P}$ be a family of pairs of subsets of a set $X$, then for any $x$, write $(\mathcal{P})_x=\{(P_1,P_2)\in\mathcal{P}: x\in P_2\}$, and we say that $\mathcal{P}$ is point-$<\!\!\kappa$ if for every $x$ in $X$ the set $(\mathcal{P})_x$ has size $<\!\!\kappa$. Dow, Junnila and Pelant \cite{DJP} showed that a compact space is Eberlein compact if and only if it has a $\omega$-point finite expandable network.

Point networks are the natural local version of expandable networks.
A point network (respectively, strong point network) for $X$ is a collection $\mathcal{W} = \{ \mathcal{W}(x) : x \in X\}$ where each $\mathcal{W}(x)$ is a collection of subsets of $X$ containing $x$ such that if for every point $x$  of $X$ contained in an open set $U$  there is an open $V$ containing $x$ and contained in  $U$ such that for every $y$ in $V$ there is a $W \in \mathcal{W}(y)$ such that $x \in W \subseteq U$ (respectively, $x \in W \subseteq V$).
We note that  we can take $U$ and $V$ to be basic. Point networks are also known as `condition (F)', and as the `Collins--Roscoe structuring mechanism' after the authors who introduced them \cite{CR}. The term `point network' was suggested by Gruenhage.

If $\mathcal{Q}$ is a property that can be reasonably applied to a family of subsets of a space then we say a space $X$ has  a $\mathcal{Q}$ (strong) point network if it has a (strong) point network $\mathcal{W}$  such that for each $\mathcal{W}(x)$ the property $\mathcal{Q}$ holds.

Let $\mathcal{W}$ be a (strong) point network  for a space $X$.  We can further structure the (strong) point network $\mathcal{W}$ analogously to a $P$-$\mathcal{Q}$ expandable networks. Formally, let $P$ be a directed set, and $\mathcal{Q}$ be as above. We say that  $\mathcal{W}$ for $X$ is a $P$-$\mathcal{Q}$ (strong) point network if for each  point $x$ in $X$, we can write $\mathcal{W}(x)=\bigcup_{p \in P} \mathcal{W}_p(x)$, where every $\mathcal{W}_p(x)$ has property $\mathcal{Q}$, if $p \le p'$ then $\mathcal{W}_p(x) \subseteq \mathcal{W}_{p'} (x)$, and if some point $x$ is in an open set $U$, then there is an open set $V=V(x,U)$ containing $x$ and contained in $U$ and a $p=p(x,U)$ from $P$ such that if $y \in V$ then there is a $W \in \mathcal{W}_p(y)$ such that $x \in W \subseteq U$ (respectively, $x \in W \subseteq V$).  Note that a `$1$-$\mathcal{Q}$ (strong) point network' is exactly the  same as a `$\mathcal{Q}$ (strong) point network' from above.

\subsection*{Bases give Expandable Networks, give Strong Point Networks}

The proof of the next lemma demonstrating how $P$-$\mathcal{Q}$ bases give $P$-$\mathcal{Q}$ expandable networks can safely be left to the reader.
\begin{lemma}\label{b_en}
Let $\mathcal{B}$ be a base for a space $X$. Then $\mathcal{N}=\{(B,B) : B \in \mathcal{B}\}$ is an expandable  network for $X$.

Further, if $\mathcal{B} = \bigcup_{p \in P} \mathcal{B}_p$, where $\mathcal{B}_p \subseteq \mathcal{B}_{p'}$ if $p \le p'$, and each $\mathcal{B}_p$ has an order or cardinality property $\mathcal{Q}$, then $\mathcal{N}=\bigcup_p \mathcal{N}_p$, $\mathcal{N}_p \subseteq \mathcal{N}_{p'}$ if $p \le p'$, and each $\mathcal{N}_p$ has property $\mathcal{Q}$.

Thus a space with a $P$-$\mathcal{Q}$ base has a $P$-$\mathcal{Q}$ expandable network.
\end{lemma}

Now we show that $P$-$\mathcal{Q}$ expandable networks give $P$-$\mathcal{Q}$ strong point networks. It should be evident that a $P$-$\mathcal{Q}$ strong point network is a $P$-$\mathcal{Q}$ point network.
\begin{lemma}\label{en_spn}
Suppose a space $X$ has a $P$-$\mathcal{Q}$ expandable  network $\mathcal{N}$, say $\mathcal{N}=\bigcup_{p \in P} \mathcal{N}_p$, where $\mathcal{N}_p \subseteq \mathcal{N}_{p'}$ if $p \le p'$, and each $\mathcal{N}_p$ has property $\mathcal{Q}$. Then $\mathcal{W}=\{ \mathcal{W}(x) : x \in X\}$, where for each $x$ in the space $\mathcal{W}(x)=\bigcup_p \mathcal{W}_p (x)$ and $\mathcal{W}_p(x) = \{ \{x\} \cup N : (N,V) \in \mathcal{N}_p$ and $x \in V\}$, is a $P$-$\mathcal{Q}$ strong point network for $X$.
\end{lemma}
\begin{proof} Fix $x$ in $X$ and $p \le p'$ from $P$. We verify $\mathcal{W}_p(x) \subseteq \mathcal{W}_{p'}(x)$. Take any $W$ from $\mathcal{W}_p(x)$. Then $W=\{x\} \cup N$ for some $(N,V)$ in $\mathcal{N}_p$ such that $x \in V$. Since $(N,V) \in \mathcal{N}_p$ and $\mathcal{N}_p \subseteq \mathcal{N}_{p'}$ (and $x$ is still in $V$!), we see that $W=\{x\} \cup N$ is in $\mathcal{W}_{p'}(x)$.

Now take any open set $U$ containing a point $x$. Since $\mathcal{N}=\bigcup \{\mathcal{N}_p : p \in P\}$ is an expandable point network, there is a $p=p(x,U)$ such that $x \in N \subseteq U$ for some $(N,V')$ in $\mathcal{N}_p$. Let $V(x,U)=V' \cap U$, and note $x\in V \subseteq U$. Take any $y$ in $V$. Then, as $V \subseteq V'$, by definition, $W=\{y\} \cup N$ is a member of $\mathcal{W}_p(y)$, and since $y\in V$, $x \in N \subseteq U$ and $N \subseteq V'$, we clearly have  $x \in W \subseteq V$.
\end{proof}

Combining the previous two lemmas we know if a space $X$ has a $P$-$\mathcal{Q}$ base $\mathcal{B}=\bigcup \{ \mathcal{B}_p : p \in P\}$ then it has a $P$-$\mathcal{Q}$ strong network. Indeed, we can take $\mathcal{W}=\{\mathcal{W}(x) : x \in X\}$ where $\mathcal{W}(x) = \bigcup_{p \in P} \mathcal{W}_p(x)$ and $\mathcal{W}_p(x)=(\mathcal{B}_p)_x$.

\subsection*{Stability of Point Networks}

Let $\mathcal{Q}$ be a property that can be reasonably applied to families of subsets of a topological space. We say that $\mathcal{Q}$ is hereditary if whenever $A$ is a subspace of $X$ and $\mathcal{C}$ is a family of subsets of $X$ satisfying $\mathcal{Q}$, then their traces onto $A$, $\mathcal{C}_A = \{C \cap A : C \in \mathcal{C}\}$ is a family of subsets of $A$ also satisfying $\mathcal{Q}$. We say that $\mathcal{Q}$ is (finitely) productive if whenever $\mathcal{C}$ is a family of subsets of $X$ satisfying $\mathcal{Q}$ and $\mathcal{C}'$ is a family of subsets of $Y$ satisfying $\mathcal{Q}$, then the collection of subsets of $X\times Y$ given by $\{ C \times C' : C \in \mathcal{C}, C' \in \mathcal{C}'\}$ satisfies $\mathcal{Q}$. Finally we say $\mathcal{Q}$ is preserved by images if whenever $\mathcal{C}$ is a collection of subsets of a space $X$ which satisfies $\mathcal{Q}$, and $f$ is a continuous map of $X$ into a space $Y$, then $f(\mathcal{C})=\{f(C) : C \in \mathcal{C}\}$ has $\mathcal{Q}$ in $Y$.

The next proposition is straight forward. Note for claim (2) that for any directed set $P$ we have $P \tq P \times P$, which is why the natural $(P\times P)$-$\mathcal{Q}$ (strong) point network for $X \times Y$ is in fact a $P$-$\mathcal{Q}$ (strong) point network.
\begin{prop}\label{stab}

(1) Let $\mathcal{Q}$ be hereditary. If a space $X$ has a $P$-$\mathcal{Q}$ (strong) point network then so does every subspace of $X$.

(2) Let $\mathcal{Q}$ be productive. If spaces $X$ and $Y$ have a $P$-$\mathcal{Q}$ (strong) point network, then  $X\times Y$ also has a $P$-$\mathcal{Q}$ (strong) point network.
\end{prop}

Invariance under maps requires a little more work.
\begin{prop}\label{perfect_im} Let $\kappa$ be an infinite cardinal. Let $f:X\rightarrow Y$ be a closed surjection with $<\!\!\kappa$ compact fibers. Let $P$ be a $<\!\!\kappa$ directed set, and $\mathcal{Q}$ is a property preserved by taking images.

Then $Y$ has a $P$-$\mathcal{Q}$ point network provided that $X$ has a $P$-$\mathcal{Q}$ point network.
\end{prop}
\begin{proof}Let $\mathcal{W}(x)=\bigcup_{p} \mathcal{W}_p(x)$ be
a $P$-$\mathcal{Q}$ point network of $X$.
For each $y\in Y$, pick $x_y\in f^{-1}(y)$. Then for each $p$ in $P$ define
$\mathcal{W}_p(y)=\{f(W): W\in \mathcal{W}_p(x_y)\}$, and note it has property $\mathcal{Q}$. Set $\mathcal{W}(y)=\bigcup_p \mathcal{W}_p(y)$ and $\mathcal{W}_Y=\{\mathcal{W}(y) : y \in Y\}$.

We verify that $\mathcal{W}_Y$ is a $P$-$\mathcal{Q}$ point network of $Y$.
Take $y\in U$ where $U$ is an open subset in $Y$. By
the definition of $P$-$\mathcal{Q}$ point network, for any $x\in
f^{-1}(y)$, there exist $V(x, f^{-1}(U))$ and $p(x, f^{-1}(U))$ such
that for any $\hat{x} \in V(x, f^{-1}(U))$, there exists $W \in
\mathcal{W}_{p(x, f^{-1}(U))}(\hat{x})$ such that $x\in W \subseteq
U$. By hypothesis, $f^{-1}(y)$ is $<\!\!\kappa$ compact. Hence, there is a $\tau < \kappa$ such that for each $\alpha < \tau$ there exists
$x_\alpha \in f^{-1}(y)$ such that $f^{-1}(y)\subseteq
\bigcup_{\alpha < \tau} V(x_\alpha, f^{-1}(U))$. Since, by hypothesis, $P$ is $<\!\!\kappa$ directed and $\tau < \kappa$, we can pick $\hat{p}=\hat{p}(y,
U)$ satisfy $ \hat{p}\geq p(x_\alpha, f^{-1}(U)): \alpha < \tau$. As $f$ is a closed
mapping, we can pick $\hat{V}=V(y, U)$ such that $f^{-1}(y)\subseteq
f^{-1}(\hat{V})\subseteq \bigcup_{\alpha < \tau} V(x_\alpha, f^{-1}(U))$.

Take any $\hat{y}\in\hat{V}$. Then $x_{\hat{y}}\in
f^{-1}(\hat{V})\subseteq \bigcup_{\alpha < \tau} V(x_\alpha, f^{-1}(U))$. Hence
$x_{\hat{y}}\in V(x_\alpha, f^{-1}(U))$ for some $\alpha$. Therefore, there
exists $W\in \mathcal{W}_{p(x_\alpha, f^{-1}(U))}(x_{\hat{y}})$ such that
$x_\alpha\in W\subseteq f^{-1}(U)$. So we have $W\in
\mathcal{W}_{\hat{p}}(x_{\hat{y}})$ and $y=f(x_\alpha)\in f(W)\subseteq
U$ where $f(W)\in \mathcal{W}_{\hat{p}}(\hat{y})$.
\end{proof}

\begin{lemma}\label{n_stab} For any cardinal $\kappa$:

(1) the property $\mathcal{Q}=$ `$<\!\!\kappa$' (has size strictly less than $\kappa$) is  hereditary, productive and preserved by images,

(2) the property $\mathcal{Q}=$ `additively $<\!\!\kappa$-Noetherian' is hereditary and preserved by images, while

(3) the property `additively Noetherian' is productive.
\end{lemma}
\begin{proof}
Claim (1) is immediate, as is the fact that additively $<\!\!\kappa$-Noetherian is hereditary. Claim (3) was established by Nyikos in \cite{Ny}. We indicate why  additively $<\!\!\kappa$-Noetherian is preserved by images.

Fix spaces $X$ and $Y$, a collection $\mathcal{C}$ of subsets of $X$, and a map $f$ of $X$ into $Y$. Applying clause (4) of Lemma~\ref{Noeth} it is simple to see that if $f(\mathcal{C})$ is not additively $<\!\!\kappa$-Noetherian then neither is $\mathcal{C}$.
\end{proof}

Recall that in this paper all partial orders $P$ are directed (in other words, $<\!\!\aleph_0$ directed), and note that the trivial partial order $\mathbf{1}$ is (trivially!) $<\!\!\kappa$ directed for all $\kappa$. Thus we deduce:
\begin{cor}\label{perfect_i} \

(1) \ If $f: X \to Y$ is a closed surjection and $X$ has a $\mathcal{Q}$ point network then so does $Y$, whenever $\mathcal{Q}$ is closed under images.

(2) \ If $f: X \to Y$ is a perfect surjection and $X$ has a $P$-additively $<\!\!\kappa$ Noetherian point network then so does $Y$.

(3) \ If $f: X \to Y$ is a perfect surjection and $X$ has a $P$-finite point network then so does $Y$.
\end{cor}

\subsection*{Covering Properties From Point Networks}

A space is said to be $P$-meta $<\!\!\kappa$ compact if and only if every open cover has a $P$-point $<\!\!\kappa$ open refinement. Let us abbreviate `$P$-meta $<\!\!\aleph_0$ compact' to `$P$-metacompact'. Note that a space is metacompact if and only if it is $1$-metacompact, $\sigma$-metacompact if and only if it is  $\N$-metacompact, and is metaLindelof if and only if it is $1$-meta $<\!\!\aleph_1$ compact.

\begin{theorem}
If a space has a $P$-additively $<\!\!\kappa$ Noetherian point network then it   is   $P$-meta $<\!\!\kappa$ compact.
\end{theorem}

\begin{proof}
Let $\mathcal{W}=\{ \mathcal{W}(x) : x \in X\}$, where $\mathcal{W}(x)=\bigcup_{p \in P} \mathcal{W}_p(x)$, each $\mathcal{W}_p(x)$ is $<\!\!\kappa$ additively Noetherian, if $p \le q$ then $\mathcal{W}_p(x) \subseteq \mathcal{W}_q(x)$, and this whole structure is a point network for a space $X$.

We show $X$ is $P$-meta $<\!\!\kappa$ compact. To this end take any open cover of $X$ say $\mathcal{U}=\{U_\alpha : \alpha < \tau\}$. For each $\alpha$ and $p \in P$ define $V_{\alpha,p} = \bigcup \{ V(x,U_\alpha) : x \in U_\alpha \setminus \bigcup \{ U_\beta : \beta < \alpha\}, \, p(x,U_\alpha)=p\}$. Set $\mathcal{V}_{\le p} = \{ V_{\alpha, p'} : \alpha < \tau, \, p' \le p\}$, and $\mathcal{V}=\bigcup_{p \in P} \mathcal{V}_{\le p}$.

Then $\mathcal{V}$ clearly is an open refinement of $\mathcal{U}$. It is also clear that if $p \le q$ then $\mathcal{V}_{\le p} \subseteq \mathcal{V}_{\le q}$. So it remains to show that each $\mathcal{V}_{\le p}$ is point-$<\!\!\kappa$. Suppose not. Then there is a point $y$ in $X$ such that for each $\gamma < \kappa$, $y$ is in some $V(x_\gamma,U_{\alpha_\gamma})$,  where $x_\gamma \in U_{\alpha_\gamma} \setminus \bigcup \{ U_\beta : \beta < \alpha_\gamma\}$, $p(x_\gamma, U_{\alpha_\gamma})=p_\gamma \le p$ and $\delta < \gamma$ implies $\alpha_\delta < \alpha_\gamma$. Hence, for each $\gamma < \kappa$, there is a $W_\gamma$ in $\mathcal{W}_{p_\gamma}(y) \subseteq \mathcal{W}_p (y)$ such that $x_\gamma \in W_\gamma \subseteq U_{\alpha_\gamma}$. Thus, for each $\delta < \gamma < \kappa$ we have that $x_\gamma \in W_\gamma$, but $x_\gamma \notin U_{\alpha_\delta} \supseteq W_\delta$. But by Lemma~\ref{Noeth} (4) this explicitly contradicts the fact that $\mathcal{W}_p(y)$ is $<\!\!\kappa$ additively Noetherian.
\end{proof}

From Proposition~\ref{stab} and Lemma~\ref{n_stab} we see:
\begin{cor}\label{pn_cp} \

 (1) \ For any $\kappa$, a space with a $P$-additively $<\!\!\kappa$ Noetherian point network    is hereditarily  $P$-meta $<\!\!\kappa$ compact.

 (2) \ A finite product of spaces each with a $P$-additively Noetherian point network is (hereditarily) $P$-metacompact.
\end{cor}

\subsection*{For Compacta, Off Diagonal Covering Properties give Almost Subbases}

For a compact space $X$ we have shown that if $X$ has a  suitable almost subbase then it  is hereditarily $P$-meta $<\kappa$ compact, for some directed set $P$ and infinite cardinal $\kappa$. We now attempt to close the loop and prove:

\begin{theorem}\label{cpt_off_diag} Let $X$ be a compact space. If $X^2 \setminus \Delta$ is $P$-meta $<\!\!\kappa$ compact then $X$ has a $(P\times \N)$-point $<\!\!\kappa$ almost subbase.
\end{theorem}

The proof is essentially the same as that given by Garcia et al in \cite{GOO} for the case when $P=\K(M)$, which in turn was a straightforward extension of Gruenhage's original argument for the $P=\N$ case. Consequently we only sketch the proof for general directed sets $P$. We start with a lemma which is the $P$-analogue of Proposition~8 of \cite{GOO}.

\begin{lemma}\label{nice_cover} Let $X$ be a $P$-meta $<\!\!\kappa$ compact, locally compact space, and let $\mathcal{B}$ be a basis for $X$. Then $X$ has a subcover $\mathcal{B}' \subseteq \mathcal{B}$ such that $\{\cl{B} : B \in \mathcal{B}'\}$ is a $P$-point $<\!\!\kappa$ family.
\end{lemma}

The proof of this lemma is identical to that of Proposition~8 of \cite{GOO}, it suffices to replace references in that argument of `point finite' with `point $<\!\!\kappa$' and `$K \in \K(M)$'  with `$p \in P$' --- nothing specific about the partial order of set inclusion on $\K(M)$ is used.

Our Theorem~\ref{cpt_off_diag} is essentially the content of the implication `(ii) implies (i)' of Theorem~9 of \cite{GOO}. We follow their argument making the necessary adjustments.

\begin{proof} Since $X^2 \setminus \Delta$ is $P$-meta $<\!\!\kappa$ compact,  applying Lemma~\ref{nice_cover} to $\mathcal{B}=\{U \times V : U,V$ are cozero subsets of $X$ such that $\cl{U} \cap \cl{V} = \emptyset\}$ and tidying, we obtain a cover $\mathcal{P}=\{U_\gamma \times V_\gamma : \gamma \in A\}$ of $X^2 \setminus \Delta$ with the following properties:
(a) $U_\gamma$ and $V_\gamma$ are cozero sets of $X$, (b) $\cl{U_\gamma} \cap \cl{V_\gamma} = \emptyset$, (c) $\{ \cl{U_\gamma} \times \cl{V_\gamma} : \gamma \in A\}$ is $P$-point $<\!\!\kappa$ and (d) if $U \times V$ is in $\mathcal{P}$ then so is   $V \times U$.

Suppose  $X=\overline{\{p_\alpha : \alpha < \mu\}}$ where $\mu=d(X)$ (the minimal size of a dense set of $X$). Set, for each $\alpha < \mu$, $X_\alpha = \cl{\{p_\beta : \beta < \alpha\}}$ and $\mathcal{U}_\alpha = \{ \bigcap_{\gamma \in F} U_\gamma : F \subseteq A$ and $\{\cl{V_\gamma} : \gamma \in F\}$ is a finite minimal cover of $X_\alpha\}$.

Note that $\mathcal{U}_\alpha$ covers $X\setminus X_\alpha$. Then the family $\mathcal{U}=\bigcup \{\mathcal{U}_\beta : \beta < \mu\}$ is $T_0$-separating as in \cite{Gruen} (Theorem 2.2, Claim 2). Since the elements of $\mathcal{U}$ are cozero sets, and $X$ is compact, we deduce that $\mathcal{U}$ is an almost subbase for $X$.

It remains to show that $\mathcal{U}$ is a $(P \times \N)$-point $<\!\!\kappa$ family. Well we know that $A=\bigcup \{A_p : p \in P\}$ where $A_p \subseteq A_q$ whenever $p \le q$ and each family $\{\cl{U_\gamma} \times \cl{V_\gamma} : \gamma \in A_p\}$ is point $<\kappa$.

Fix $p$ in $P$ and $n$ in $\N$. For any $\alpha$, let $\mathcal{U}^p_{\alpha, n}$ be all members of $\mathcal{U}_\alpha$ whose index set $F$ has size $\le n$ and is contained in $A_p$. Let $\mathcal{U}^p_n = \bigcup \{ \mathcal{U}^p_{\alpha,n} : \alpha < \mu\}$. Then $\mathcal{U}=\bigcup \{   \mathcal{U}^p_n : p \in P, \ n \in \N\}$ and clearly if $p \le p'$, $n \le n'$ then $\mathcal{U}^p_n \subseteq \mathcal{U}^{p'}_{n'}$.

To complete the proof we will show that $\mathcal{U}^p_n$ is point-$<\!\!\kappa$. Suppose, for a contradiction, that for some point $x$ in $X$ and each $\rho < \kappa$ there are $\beta_\rho < \mu$ and sets $F_\rho \subseteq A_p$ such that $|F_\rho| \le n$, $x \in \bigcap \{ U_\gamma : \gamma \in F_\rho\}$, $X_{\beta_\rho} \subseteq \bigcup \{ \cl{V_\gamma} : \gamma \in F_\rho\}$ and $F_\rho \ne F_\sigma$ if $\rho \ne \sigma$. Tidying we can assume $\beta_{\rho} \le \beta_{\sigma}$ if $\rho \le \sigma$.

Since for every $\rho < \kappa$ we have $|F_\rho| \le n$ and all the $F_\rho$'s are different, we may assume that $\{F_\rho : \rho < \kappa\}$ form a $\Delta$-system with root $R$. Pick $y \in X_{\beta_0} \setminus \bigcup \{\cl{V_\gamma} : \gamma \in R\}$. Then for each $\rho$ there is a $\delta(\rho) \in F_\rho \setminus R$ with $y \in \cl{V_{\delta(\rho)}}$. But then $(x,y) \in \bigcap \{ U_{\delta(\rho)} \times \cl{V_{\delta(\rho)}} : \rho < \kappa\}$ and for all $\rho$, $\delta(\rho)$ is in $A_p$, which contradicts $\{ \cl{U_\gamma} \times \cl{V_\gamma} : \gamma \in A_p\}$ being point-$<\!\!\kappa$.
\end{proof}

\subsection*{Characterizations, and an Application}

Recall that a directed set is not countably directed (every countable subset has an upper bound) if and only if $P \tq \N$. In particular $\N \tq \N \times \N$. Hence if $P$ is not countably directed then $P \te P \times \N^n$ for every $n$ in $\N$. Applying this fact and Theorem~\ref{as_is_embed}, Theorem~\ref{as_b}, Lemma~\ref{b_en}, Lemma~\ref{en_spn}, Corollary~\ref{pn_cp}, and Theorem~\ref{cpt_off_diag}, we establish the equivalence of claims (1) and (2) of the following theorem, and then the implications (2) implies (3), (3) implies (4), (4) implies (5), and (6) implies (2) (the implication (5) implies (6) being trivial).
\begin{theorem}\label{main1} Let $X$ be compact. Let $P$ be a directed set which is not countably directed. Then the following are equivalent:

(1) $X$ embeds in a $C_p(X(\Filt))$ where $X(\Filt)$ has a $P$-ordered compact cover,

(2) $X$ has a $P$-point finite almost subbase,

(3) $X$ has a $P$-point additively Noetherian base,

(4) $X$ has a $P$-point additively Noetherian expandable network,

(5) $X$ has a $P$-additively Noetherian strong point network, and

(6) $X$ has a $P$-additively Noetherian point network.
\end{theorem}

Essentially only one directed set $P$ is not covered by the above theorem. To see this take any countably directed $P$ and suppose (1) above holds: $X$ embeds in a $C_p(X(\Filt))$ where $X(\Filt)$ has a $P$-ordered compact cover. Since $P$ is countably directed any countable subset of $X(\Filt)$ must be contained in one of the elements of the $P$-ordered compact cover. It follows that $X(\Filt)$ is countably compact. As $X(\Filt)$ is paracompact, we see that $X(\Filt)$ is compact. So we might as well take $P=1$, and observe that $X(\Filt)=A(\kappa)$ for some $\kappa$. Let us deal with this remaining case of $P=1$.

\begin{theorem}\label{ch_Eb}
Let $X$ be compact. Then the following are equivalent:

(0) $X$ is Eberlein compact,
(1) $X$ embeds in a $C_p(A(\kappa))$,
(2) $X$ has a $\N$-point finite almost subbase,
(3) $X$ has a $\N$-point additively Noetherian base,
(4) $X$ has a $\N$-point additively Noetherian expandable network,
(5) $X$ has a $\N$-additively Noetherian strong point network, and
(6) $X$ has a $\N$-additively Noetherian  point network.
\end{theorem}
\begin{proof}
The equivalence of (0) and (1) --- taking the definition of an Eberlein compact space to be one homeomorphic to a weakly compact subspace of a Banach space --- is due to \cite{AL}. Taking $P=1$ and recalling that $\N \tq \N \times \N$, by the same argument as above, we have (1) implies (2), (2) implies (3), (3) implies (4), (4) implies (5), (5) implies (6), and (6) implies (2). Theorem~\ref{as_is_embed} tells us that $X$ embeds in a $C_p(X(\Filt))$ where $X(\Filt)$ is $\sigma$-compact (has an $\N$-ordered compact cover). But it is well-known and not difficult to verify that $C_p(X(\Filt))$ is then homeomorphic (not linearly) to a subspace of some $C_p(A(\kappa))$. So we get back to (1).
\end{proof}

In addition to Amir and Lindenstruass' proof of the equivalence of (0) and (1) in Theorem~\ref{ch_Eb}, we should also mention that the equivalence of (1) and (2) is essentially Rosenthal's theorem \cite{Ro}, and Junnila \cite{Jun} established equivalence of (2) and (3).

For completeness we also highlight the cases of Theorem~\ref{main1} for $P=\N^\N$ and $P=\K(M)$ for some separable metrizable $M$.

\begin{theorem}\label{ch_T}
Let $X$ be compact. Then the following are equivalent:

(0) $X$ is Talagrand compact,
(1) $X$ embeds in a $C_p(X(\Filt))$ where $X(\Filt)$ has a $\N^\N$-ordered compact cover,
(2) $X$ has a $\N^\N$-point finite almost subbase,
(3) $X$ has a $\N^\N$-point additively Noetherian base,
(4) $X$ has a $\N^\N$-point additively Noetherian expandable network,
(5) $X$ has a $\N^\N$-additively Noetherian strong point network, and
(6) $X$ has a $\N^\N$-additively Noetherian  point network.
\end{theorem}

The equivalence of (0) and (1) above is due to \cite{Merc, Sokol}.

\begin{theorem}\label{ch_G}
Let $X$ be compact. Then the following are equivalent:

(0) $X$ is Gulko compact,
(1) $X$ embeds in a $C_p(X(\Filt))$ where $X(\Filt)$ has a $\K(M)$-ordered compact cover for some non-compact, separable metrizable $M$,

Further, for a fixed non-compact, separable metrizable space $M$, the following are equivalent:
(2) $X$ has a $\K(M)$-point finite almost subbase,
(3) $X$ has a $\K(M)$-point additively Noetherian base,
(4) $X$ has a $\K(M)$-point additively Noetherian expandable network,
(5) $X$ has a $\K(M)$-additively Noetherian strong point network, and
(6) $X$ has a $\K(M)$-additively Noetherian  point network.
\end{theorem}

The equivalence of (0) and (1) above is due to \cite{Merc, Sokol}. We observe that for a separable metrizable space $M$, the ordered set $\K(M)$ is countably directed if and only if $M$ is compact (in which case $\K(M)\te 1$).

\medskip

The equivalence of condition (6) with the others in the theorems above has a pleasant application. Recall that $P$-additively Noetherian networks are preserved by perfect images (Corollary~\ref{perfect_i}) to deduce:

\begin{theorem}\label{main_app}
Let $X$ be compact. Let $P$ be a directed set which is either not countably directed or $P\te 1$. If $X$ satisfies any of the equivalent conditions of Theorem~\ref{main1}, and $Y$ is the continuous image of $X$ then $Y$ satisfies the same conditions.
\end{theorem}

\begin{cor}
The continuous image of an Eberlein (respectively, Talagrand or Gulko) compact space is again Eberlein (respectively, Talagrand or Gulko).
\end{cor}

We note there is nothing new about this result, but emphasize the uniform nature of the proofs for Eberlein, Talagrand and Gulko compacta, and mention again the simplicity of the proof in contrast to Rudin's argument for Eberlein compacta.

\subsection*{A Corson Counter-Example}

All of the ingredients used in the proof of Theorem~\ref{main1} remain true when we replace `compact' with `$<\!\!\kappa$-compact' and `finite' with `$<\!\!\kappa$'. Except one. In Lemma~\ref{n_stab} (3) we only assert the productivity of `additively Noetherian' and \emph{not} the productivity of `additively $<\!\!\kappa$ Noetherian' for uncountable $\kappa$.

We give here an example of a compact space which shows that no analogue of Theorems~\ref{ch_Eb}, \ref{ch_T} or~\ref{ch_G} holds for Corson compact spaces. Noting that a compact space $X$ is Corson compact if and only if it embeds in some $C_p(L(\kappa))$, and $L(\kappa)$ has a $1$-ordered cover by $<\!\!\aleph_1$-compact sets, we see that the weakest conjecture is that a compact space is Corson if and only if it has a point additively $\aleph_0$-Noetherian base.

\begin{exam}
Let $X$ be the Double Arrow space. Then $X$ is compact, has a (point) additively $\aleph_0$-Noetherian base but not Corson compact.
\end{exam}
\begin{proof}
The Double Arrow space $X$ is hereditarily separable but non-metrizable, so far from Corson compact. It is also hereditarily Lindelof. Let $\mathcal{B}$ be the collection of all open subsets of $X$. It is a base for $X$. The hereditary Lindelof property immediately shows that condition (3) of Lemma~\ref{Noeth} holds for $\mathcal{B}$ with $\kappa=\aleph_1$, so it is additively $\aleph_0$-Noetherian, and {\it a fortiori} $\mathcal{B}$ is point additively $\aleph_0$-Noetherian.
\end{proof}

\subsection*{Why Calibre $(\omega_1,\omega)$ is Critical}

Theorem~\ref{main1} applies to \emph{all} directed sets $P$ which are not countably directed (and Theorem~\ref{ch_Eb} essentially covers all the remaining directed sets). A natural question is to ask for which $P$ is Theorem~\ref{main1} `interesting'? It is striking that this admittedly vague question has a clear answer: Theorem~\ref{main1} is interesting if and only if $P$ has calibre $(\omega_1,\omega)$.

A directed set $P$ has calibre $(\kappa^+, \kappa)$ if for every $\kappa^+$ sized subcollection $A$ of $P$ there is a $\kappa$ sized subset $A_0$ of $A$ which has an upper bound in $P$. We note that $\K(M)$ (and hence $\N^\N \te \K(\N^\N)$) has calibre $(\omega_1, \omega)$.

Informally, the following theorem says that in the case when $P$ does not have calibre $(\omega_1,\omega)$ then far too many spaces satisfy the conditions of Theorem~\ref{main1} for it to be of interest. But --- provided we agree that Corson compacta are `nice', and surely we do --- the  next result also says that when $P$ does have calibre $(\omega_1,\omega)$ then compact spaces satisfying the conditions of Theorem~\ref{main1} are `nice'.

\begin{theorem}\label{cal_crit}
Let $P$ be a directed set.

(1) If $P$ is not calibre $(\omega_1,\omega)$ then \emph{every} space of weight $\le \omega_1$ has a $P$-point finite base (and hence a $P$-point finite almost subbase, $P$-point additively Noetherian base/expandable network, and a $P$-additively Noetherian (strong) point network).

(2) If $P$ has calibre $(\omega_1,\omega)$ and $X$ is a compact space with a $P$-finite additively Noetherian point network (or a $P$-point finite almost subbase et cetera) then $X$ is Corson compact.
\end{theorem}
\begin{proof}[Proof (1)]
Take any space $X$ with weight no more than $\omega_1$. Fix a base $\mathcal{B}$ for $X$ where $|\mathcal{B}| \le \omega_1$. List $\mathcal{B}$, with repeats if necessary, as $\mathcal{B}=\{B_\alpha : \alpha < \omega_1\}$. Then writing $\mathcal{B} = \bigcup \{\mathcal{B}_F : F \in [\omega_1]^{<\omega}\}$ where $\mathcal{B}_F = \{ B_\alpha : \alpha \in F\}$ demonstrates that $X$ has a $[\omega_1]^{<\omega}$-(point) finite base.

But recall that a directed set $P$ does not have calibre $(\omega_1,\omega)$ if and only if $P \tq [\omega_1]^{<\omega}$. Hence if $P$ does not have calibre $(\omega_1,\omega)$ then every space of weight $\le \omega_1$ has a $P$-point finite base, as claimed.
\end{proof}

To prove the second claim we will use the following lemma of independent interest.
\begin{lemma}\label{use_calibre} Let $P$ be a directed set with calibre $(\kappa^+, \kappa)$, and let $\mathcal{C}=\{ \mathcal{C}_p : p \in P\}$ be a $P$-ordered collection.

(a) If each $\mathcal{C}_p$ has size $<\!\!\kappa$, then $\bigcup \mathcal{C}$ has size $\le\! \kappa$.

(b) If each $\mathcal{C}_p$ is additively $<\!\!\kappa$ Noetherian, then $\bigcup \mathcal{C}$ is additively $\le\! \kappa$ Noetherian.
\end{lemma}
\begin{proof}
First suppose, for a contradiction, that $\bigcup \mathcal{C}$ has size $>\kappa$ but each $\mathcal{C}_p$ has size $<\!\!\kappa$. Then by transfinite induction we can find, for every $\alpha < \kappa^+$, points $x_\alpha$ and $p_\alpha$ from $P$ such that $C_\alpha \in \mathcal{C}_{p_\alpha}$ but $C_\alpha \notin \mathcal{C}_{p_\beta}$ for any $\beta < \alpha$. By  calibre $(\kappa^+, \kappa)$ applied to $\{p_\alpha : \alpha < \kappa^+\}$, there is a $\kappa$ sized subset $S_0$ of $\kappa^+$ and a $p_0$ in $P$ such that $p_0$ is an upper bound of $\{p_\alpha : \alpha \in S_0\}$. But now we see that for every $\alpha$ in $S_0$, the element $C_\alpha$ is in $C_{p_\alpha}$ which is a subset of $\mathcal{C}_{p_0}$, and so $\mathcal{C}_{p_0}$ must have size at least $\kappa$, contradicting our assumption on the size of the $\mathcal{C}_p$'s.

To prove (b), suppose, for a contradiction, that $\bigcup \mathcal{C}$ is not additively $\le\! \kappa$ Noetherian. By lemma \ref{Noeth} (4), we can find, for all $\alpha<\kappa^+$, sets $C_\alpha$ from $\bigcup \mathcal{C}$ and $x_\alpha$ such that for all $\beta<\alpha<\kappa^+$ the point $x_\alpha$ is in $C_\alpha$ but not in $C_\beta$. For each $\alpha<\kappa^+$, we choose $p_\alpha\in P$ such that $C_{\alpha}\in\mathcal{C}_{p_\alpha}$. Similarly, by  calibre $(\kappa^+, \kappa)$ applied to $\{p_\alpha : \alpha < \kappa^+\}$, there is a $\kappa$ sized subset $S_0$ of $\kappa^+$ and a $p_0$ in $P$ such that $p_0$ is an upper bound of $\{p_\alpha : \alpha \in S_0\}$. But now we see that for every $\alpha$ in $S_0$, the element $x_\alpha$ is in $C_{p_\alpha}$ which is a subset of $C_{p_0}$, and also, for all $\alpha, \beta\in S_0$ with $\beta<\alpha$, the point $x_\alpha$ is in $C_\alpha$ but not in $C_\beta$. Applying lemma \ref{Noeth} (4) again,  $C_{p_0}$ is not additively $<\kappa$ Noetherian, contradicting our assumption on the $\mathcal{C}_p$'s.
\end{proof}

\begin{proof}[Proof (2)] Suppose $P$ has calibre $(\omega_1,\omega)$, and  $X$ is a  compact space with a $P$-additively Noetherian point network. Then $Z=X^2$ also has a $P$-additively Noetherian point network, say $\mathcal{W}=\{\mathcal{W}(z):z\in Z\}$ where $\mathcal{W}(z) = \bigcup \{\mathcal{W}_p(z) : p \in P\}$. Then for each $z$ in $Z$ we see $\mathcal{W}(z)$ is a $P$-ordered collection of additively Noetherian sets, and so --- by Lemma~\ref{use_calibre} --- is $\aleph_0$ additively Noetherian. Hence $Z=X^2$ is hereditarily metaLindelof. And a compact space is Corson compact if and only if its square is hereditarily metaLindelof.
\end{proof}

\section*{Constructing Small Expandable Networks, and Point Networks}

\subsection*{The Problem}
Our key theorem (Theorem~\ref{main1}) states that for a compact space $X$ and directed set $P \tq \N$, the space  $X$ has a $P$-point finite almost subbase if and only if it has a $P$-point  additively Noetherian base (or expandable network), if and only if it has a $P$-additively Noetherian (strong) point network. There is an asymmetry here between the situation with almost subbases compared with bases, expandable networks and (strong) point networks. While finite families are additively Noetherian, the converse is not true. The natural problem is whether every compact space with a $P$-point finite almost subbase has a $P$-point finite base or expandable network, and a $P$-finite (strong) point network.

The problem has a negative solution for the `base' version. Any space with a $\N$-point finite base (or even point countable base) is clearly first countable. But $A(\omega_1)$, has a $\N$-point finite almost subbase but is not first countable.

In this section we give solutions to the remaining parts of the problem which while incomplete cover the most important cases (when $P=\N$, $\N^\N$ or $\K(M)$).

We first show that provided $P$ has calibre $(\kappa^+,\kappa)$ then a space has a $P$-point $<\!\!\kappa$  expandable network if and only if it has a $P$-point $<\!\!\kappa$  strong point network. So only two questions remain.

For general directed sets $P$ it is not clear to the authors that having a $P$-point finite almost subbase implies the existence of a
$P$-point finite expandable network (or even a point network). But the partial orders motivating us, namely $\N$, $\N^\N$ and $\K(M)$, are not only partial orders but also have natural topologies (discrete, product and Vietoris, respectively) which interact nicely with their orders.
So we briefly move away from point networks structured by a partial order and look at point networks structured by a topological space. We then discuss topological directed sets, and their connection with topological point networks. These preliminaries combine in the proof of the main result of this section, Theorem~\ref{main2}, which says that for suitable topological directed sets $P$ (including our motivating examples), every space with a $P$-point finite almost subbase has a $(P\times \N)$-point finite expandable network.

\subsection*{Strong Point Networks give Expandable Networks}

\begin{theorem}\label{exp_is_st_pt_net}  Let $X$ be a space. Let $P$ be a directed set with calibre $(\kappa^+, \kappa)$. Then the following are equivalent:

(1) $X$ has a $P$-point-$<\!\!\kappa$ expandable network, and

(2) $X$ has a $P$-$<\!\!\kappa$ strong point network.
\end{theorem}

\begin{proof} (1) implies (2) is an instance of Lemma~\ref{en_spn} (and does not require the calibre restriction on $P$).

We prove (2) implies (1). By hypothesis $X$ has a $P$-$<\!\!\kappa$ strong point network $\mathcal{W}=\{\mathcal{W}(x) : x \in X\}$, so we can write each $\mathcal{W}(x)=\bigcup_{p \in P} \mathcal{W}_p(x)$, where every $\mathcal{W}_p$ has size $<\!\!\kappa$, if $p \le p'$ then $\mathcal{W}_p(x) \subseteq \mathcal{W}_{p'}(x)$, and if a point $x$ is in an open set $U$, then there is an open set $V=V(x,U)$ containing $x$ and contained in $U$, and a  $p= p(x,U)$ in $P$ such that whenever $y \in V$ then there is a $W$ in $\mathcal{W}_p(y)$ with $x \in W \subseteq V$.

We construct by transfinite induction for all $\alpha \in \kappa^+$,  $p$ in $P$ and points $x$ of $X$:  collections $\mathcal{A}_{\alpha}$ and  $\mathcal{A}_{\alpha, \le p}$ of pairs of points and open neighborhoods, a subset $X_\alpha$ of $X$ and a closed subset $C_{\alpha}^x$ of $X$ as follows. Let $C_{\alpha}^x=\overline{ \{y : \exists (y,V) \in \bigcup_{\beta < \alpha} \mathcal{A}_{\beta}, x \in V\}}$, $X_\alpha = \{ x : x \in C_{\alpha}^x\}$, $\mathcal{A}_{\alpha}$  is a maximal subcollection of $\{(x,V(x,X \setminus C_{\alpha}^x)) : x \notin X_{\alpha}\}$ such that if $(x_1,V_1)$ and $(x_2,V_2)$ are in the collection then either $x_1 \notin V_2$ or $x_2 \notin V_1$, and $\mathcal{A}_{\alpha, \le p} = \{(x,V) \in \mathcal{A}_{\alpha} : p(x,X \setminus C_{\alpha}^x) \le p\}$.
Let $\mathcal{A}_{\le p} = \bigcup_{\alpha} \mathcal{A}_{\alpha, \le p}$ and $\mathcal{A}=\bigcup_p \mathcal{A}_{\le p}$.

Now set $\mathcal{N}_p = \{ (V \cap W, V) : \exists (x',V) \in \mathcal{A}_{\le p} \text{ and } W \in \mathcal{W}_p (x')\}$ and let $\mathcal{N} = \bigcup_p \mathcal{N}_p$. Clearly, if $p \le q$ then $\mathcal{N}_p \subseteq \mathcal{N}_q$. We will show that $\mathcal{N}$ is an expandable network, and each $\mathcal{N}_p$ is point-$<\!\!\kappa$.

Fix $p$. Observe that $\mathcal{N}_p$ is point-$<\!\!\kappa$ provided $\mathcal{A}_{\le p}$ is point-$<\!\!\kappa$, in the sense that for any point $y$ the set $(\mathcal{A}_{\le p})_y:=\{ (x',V) \in \mathcal{A}_{\le p} : y \in V\}$ has size $<\!\!\kappa$. To establish $\mathcal{A}_{\le p}$  point-$<\!\!\kappa$ we define an injection from $(\mathcal{A}_{\le p})_y$ into $\mathcal{W}_p(y)$, which we know has size $<\!\!\kappa$. Take any $(x',V)$ in $(\mathcal{A}_{\le p})_y$. So $y \in V$, and by definition of a $P$-ordered strong point network there is a $W_{(x',V)}$ in $\mathcal{W}_p(y)$ such that $x' \in W_{(x',V)} \subseteq V$.

Let us show that $(x',V) \mapsto W_{(x',V)}$ is the desired injection.
Well suppose $(x_1,V_1)$ and $(x_2,V_2)$ are in $(\mathcal{A}_{\le p})_y$. Let $W_i=W_{(x_i,V_i)}$ for $i=1,2$. Two cases arise. Suppose first both pairs are in some $\mathcal{A}_{\alpha, \le p}$. Without loss of generality we can suppose $x_1 \notin V_2$.  Since $x_1 \notin V_2$, we see $W_1 \not\subseteq W_2$, and in particular $W_1 \ne W_2$. Now suppose $(x_1,V_1) \in \mathcal{A}_{\alpha, \le p}$ and $(x_2,V_2) \in  \mathcal{A}_{\beta, \le p}$ where $\alpha < \beta$.  If $W_1=W_2$ then $x_2 \in V_1$ and $x_1 \in V_2$. But this leads to a contradiction because $x_2 \in V_1 = V(x_1, X \setminus C_{\alpha}^{x_1})$ and $C_{\beta}^{x_2} \supseteq \{ z : \exists (z,V) \in \mathcal{A}_{\alpha}$ and $x_2 \in V\} \ni x_1$, so $x_1 \notin V(x_2, X \setminus C_{\beta}^{x_2})=V_2$. Thus $W_1 \ne W_2$ in the second case as well as the first, and our map is indeed injective.

It remains to show that $\mathcal{N}$ is an expandable network. To this end take any point $x$ in an open set $U$. We will show that $x$ is in some $X_\alpha$. If so then, by definition of $X_\alpha$ and $C_{\alpha}^x$, the neighborhood $V(x,U)$ of $x$ must meet $\{x' : \exists (x',V) \in \bigcup_{\beta < \alpha} \mathcal{A}_{\beta}, x \in V\}$, and there is a point $x'$ in $V(x,U)$, element  $p_0$, and $(x',V)$ in $\mathcal{A}_{\alpha, \le p_0}$ with $x \in V$. Let $p$ be an upper bound of $p_0$ and  $p(x,U)$. Then $(x',V)$ is in $\mathcal{A}_{\le p}$ and there is a $W$ in $\mathcal{W}_p (x')$ such that $x \in W \subseteq U$. Now we have that $x \in W \cap V \subseteq U$ and $(W \cap V,V)$ is in $\mathcal{N}_p$, as desired.

Suppose, for a contradiction, that for all  $\alpha < \kappa^+$ the point $x$ is not in $X_\alpha$. Then for all $\alpha$ the point $x$ is not in $C_{\alpha}^x$. As $\mathcal{A}_{\alpha}$ is a maximal subcollection of $\{ (x',V(x',X \setminus C_{\alpha}^{x'}) : x' \notin C_{\alpha}^{x'}\}$, there must be a pair $(x',V)$ in $\mathcal{A}_{\alpha}$ such that $x \in V$. However $\mathcal{A}=\bigcup_p \mathcal{A}_{\le p}$, and for each $p$ in $P$ the set $(\mathcal{A}_{\le p})_x$ has size $<\!\!\kappa$, so recalling that $P$ has calibre $(\kappa^+,\kappa)$, applying Lemma~\ref{use_calibre} we see that there can only be $\le\!\kappa$ many $\mathcal{A}_\alpha$ containing a pair $(x',V)$ such that $x \in V$, contradiction.
\end{proof}

\begin{cor} \
\begin{description}

\item[($P=\N, \kappa=\aleph_0$)] A space has a $\sigma$-point finite expandable network if and only if it has a $\sigma$-finite strong point network.

\item[($P=\K(M), \kappa=\aleph_0$)] A space has a $\K(M)$-point finite expandable network if and only if it has a $\K(M)$-finite strong point network.

\item[($P=\N^\N, \kappa=\aleph_0$)] A space has a $\N^\N$-point finite expandable network if and only if it has a $\N^\N$-finite strong point network.
\item[($P=1, \kappa=\aleph_1$)] A space has a point countable expandable network if and only if it has a countable	 strong point network.
\end{description}
\end{cor}

\subsection*{Topological Point Networks}

\begin{defn}
Let $X$ be a space. Given a space $Z$ and  property $\mathcal{Q}$, we say $\mathcal{W}=\{\mathcal{W}(x) : x \in X\}$ is a $Z$-$\mathcal{Q}$ (strong) point network (in the topological sense) if for each $x$ in $X$ we can write $\mathcal{W}(x)=\bigcup_{z \in Z} \mathcal{W}_z(x)$,  and:

(1) if $x$ is in open $U$, then there is an open $V=V(x,U)$ containing $x$ and contained in $U$, and a $z=z(x,U)$ in $Z$ such that if $y \in V$ then there is a $W$ in $\mathcal{W}_z(y)$ such that $x \in W \subseteq U$ (respectively, $W \subseteq V$ for the strong version),

(2) for every $x$ and $z$ there is an open $T$ around $z$ such that $\mathcal{W}_T (x)= \bigcup \{ \mathcal{W}_{z'} (x) : z' \in T\}$ has property $\mathcal{Q}$.
\end{defn}

\begin{lemma}\label{top_2_ord}
Let a space $X$ have a $Z$-finite (strong) point network where $Z$ is a space with $P$-ordered compact cover. Then $X$ has a $P$-finite (strong) point network.
\end{lemma}
\begin{proof}
Let $\mathcal{W}=\{\mathcal{W}(x) : x \in X\}$ where $\mathcal{W}(x)=\bigcup_{z \in Z} \mathcal{W}_z(x)$ be as in the definition of a
 $Z$-finite (strong) point network, and let $Z=\bigcup_{p \in P} K_p$ be a $P$-ordered compact cover of $Z$.

Fix $x$ in $X$. For each $z$ in $Z$ fix the open set $T_z$ given by condition (1). Fix  $p$ in $P$. As $K_p$ is compact some finite subcollection of  $\{T_z : z \in Z\}$ covers $K_p$. By condition (2) we see that $\mathcal{W}_{K_p}(x)$ is finite (as a subset of a finite union of finite sets).

Let $\hat{\mathcal{W}}=\{ \hat{\mathcal{W}} (x) : x \in X\}$ where  $\hat{\mathcal{W}} (x)=\bigcup_p \hat{\mathcal{W}}_p (x)$ and $\hat{\mathcal{W}}_p(x) = \mathcal{W}_{K_p}(x)$. We have just shown that each $\hat{\mathcal{W}}_p (x)$ is finite, and then it is easy to check $\hat{\mathcal{W}}$ is the desired  $P$-finite (strong) point network.
\end{proof}

\subsection*{Topological Directed Sets}
Topological point networks are relevant here because the directed sets we are most interested in carry natural topologies which interact with the order. We call a directed set with a topology a \emph{topological directed set}. Any directed set becomes topological with the discrete topology. More usefully, $\N^\N$ with the product topology and any $\K(M)$ with the Vietoris topology (where $M$ is separable metrizable) are separable metrizable topological directed sets in which all down sets are compact and every convergent sequence has an upper bound.

In the next two lemmas we see how the topology of a topological directed set can influence the order, and vice versa.
\begin{lemma}\label{get_cal}
Let $P$ be a topological directed sets which is Frechet-Urysohn, has countable extent, and every convergent sequence has an upper bound. Then $P$ has calibre $(\omega_1,\omega)$.
\end{lemma}

\begin{proof}
Take any uncountable subset $S$ of $P$. Since $P$ is has countable extent, $S$ is not a closed discrete subset, so there is a $p$ in $P$ which is in the closure of $S \setminus \{p\}$. As $P$ is Frechet-Urysohn, there is a sequence on $S\setminus \{p\}$ converging to $p$. By hypothesis on $P$, this {\em infinite} sequence has an upper bound. Thus some infinite subset of $S$ has an upper bound, as required for   calibre $(\omega_1, \omega)$ to hold.
\end{proof}

\begin{lemma}\label{base_finite} Let $P$ be a topological directed set which is first countable, and every convergent sequence has an upper bound. Let $\mathcal{C}$ be a $P$-finite collection of subsets of a space $X$. Then for any $p\in P$ there exists an open neighborhood $T$ of $p$ such that $\bigcup \{\mathcal{C}_q: q\in T\}$ is finite.
\end{lemma}

\begin{proof} Fix $p$ in $P$. Since $P$ is first countable we can fix $\{B_n:n \in \N\}$  a local base at $p$.
Suppose, for contradiction, that  for any open neighborhood $T$ of $p$ the set $\bigcup \{\mathcal{C}_q: q\in T\}$ is infinite.  Then for each $n$, we can find $p_n\in B_n$ and  $C_{n}\in \mathcal{C}_{p_n}\setminus \{C_i: i<n\}$. Note that the sequence $\{p_n: n \in \N\}$  converges to $p$. By hypothesis, there is $p_0\in P$ such that $p_n\leq p_0$ for all $n$. Since $\mathcal{C}$ is $P$-ordered, $\{C_n: n \in \N\}$ is an infinite subcollection of $\mathcal{C}_{p_0}$ which is a contradiction.
\end{proof}

Now we connect topological directed sets and the topological point networks of the previous section.
\begin{lemma}\label{top_to_order}
Let $P$ be a topological directed set in which down sets are compact. Then if a space $X$ has a $P$-finite (strong) point network in the topological sense then it also has a $P$-finite (strong) point network in the order sense.
\end{lemma}
\begin{proof}
Since the down sets are compact and $P = \bigcup_{p \in P} (\down{p})$, the space $P$ has a $P$-ordered compact cover, so we can apply Lemma~\ref{top_2_ord}.
\end{proof}

\subsection*{Some Almost Subbases give Small Expandable Networks}

\begin{theorem}\label{main2}
Let $P$ be a topological directed set which is separable metrizable, down sets are compact and every convergent sequence has an upper bound.

If a space $X$ has a $P$-point finite almost subbase then it has a $(P \times \N)$-point finite expandable network.
\end{theorem}

\begin{proof}
Let $\mathcal{B}$ be a countable base for $P$.  Give $Q=[\mathcal{B}]^{<\omega}$ the discrete topology, and order it by $\subseteq$. Note that $P \times \N \times Q$ is a topological directed set (with product partial order and topology) with the same properties hypothesized for $P$.

We will show that $X$ has a $(P \times \N \times Q)$-finite topological strong point network. Then, by Lemma~\ref{top_to_order}, $X$ has a $(P \times \N \times Q)$-finite strong point network (in the order sense).
Lemma~\ref{get_cal} permits us to apply Lemma~\ref{exp_is_st_pt_net}, and, after noting that $P \times \N \times Q \te P \times \N$,  the proof is complete.

Let  $\alpha=\bigcup\{ \alpha_p: p\in P\}$ be an almost subbase for $X$ where each $\alpha_p$ is point-finite and $\alpha_p\subseteq \alpha_{p'}$ if $p\leq p'$. We assume, without loss of generality, that each $\alpha_p$ is closed under finite intersections.

For any $S \subseteq P$ write $\alpha_S = \bigcup \{\alpha_p :p \in S\}$. Fix an enumeration of $\mathcal{B}$. Because $P$ has the relevant order and topological properties, we know (Lemma~\ref{base_finite}) that for every $p$ in $P$ and $x$ in $X$ there is a $B$ in $\mathcal{B}$ containing $p$ such that $(\alpha_B)_x$ is finite. Let $B(p,x)$  be the first member in the enumeration of $\mathcal{B}$  such that $p\in B(p,x)$ and $(\alpha_{B(p,x)})_x$ finite.

Fix $x\in X$.  Fix $p\in P$, $m \in \N$ and $\mathcal{B}_0$ in $Q$. Let $\mathcal{W}_{p,m,\mathcal{B}_0}(x)$ be all (finite) intersections of: (i) $A\in (\alpha_{B(p,x)})_x$, (ii) sets of the form $A\setminus U_{2n-1}(A')$ where $A$ and $A'$ are in $(\alpha_{B(p,x)})_x$, $n\leq m$ and $x\notin U_{2n-1}(A')$, and (iii) $W_{B,x}=X\setminus \bigcup \{A\in \alpha_B: x\notin A\}$ for $B$ in $\mathcal{B}_0$. Note that $\mathcal{W}_{p,m,\mathcal{B}_0}(x)$ is finite. Let $\mathcal{W}(x) =\bigcup \{\mathcal{W}_{p,m,\mathcal{B}_0}(x): p \in P, \ m\in \N \text{ and }\mathcal{B}_0 \in Q\}$.

We  show that $\mathcal{W}=\{\mathcal{W}(x):x\in X\}$  is a $(P \times \N \times Q)$-finite topological strong point network. To do so we need to check (1) and (2) in the definition.

\noindent \emph{For (1):}
Take any point $x$ in an open set $T$. Since $\alpha\cup \{X \setminus U_{2n-1}(A) : A \in \alpha\}$ is a subbase and $P$ is directed, there is a $p_0\in P$, $M\in \N$, $A_0 \in \alpha_{p_0}$, $A_1, \ldots , A_k$ in $\alpha_{p_0}$, $n_1, \ldots , n_k \le M$ such that $x \in A_0 \cap \bigcap_{i=1}^k ( X \setminus U_{2n_i-1}(A_i)) \subseteq T$.

Let $p(x,T)=(p_0,M,\{B(p_0,x)\})$, and
\[V=V(x,T) = \bigcap (\alpha_{B(p_0,x)})_x \cap \bigcap_{i=1}^k (X \setminus U_{2n_i-1}(A_i) ).\]
Then $x$ is in $V$, which is open,  and $V \subseteq T$.

Take any $y$ in $V$. Then $y$ is in $(\alpha_{B(p_0,x)})_x$, so $(\alpha_{B(p_0,x)})_y \supseteq (\alpha_{B(p_0,x)})_x$. It follows that $ x \in W_{B(p_0,x),x} \subseteq W_{B(p_0,x),y}$.

Let $V_0 = \bigcap (\alpha_{B(p_0,x)})_x$. Relabelling if necessary, we can let $V_1=\bigcap_{i=1}^p (A_0 \setminus U_{2n_i-1}(A_i))$ and
$V_2=\bigcap_{i=p+1}^k (X \setminus U_{2n_i-1}(A_i))$, where $y$ is in $A_i$ for $i \le p$ but $y \notin A_i$ for $i>p$. So $V=V_0 \cap V_1 \cap V_2$.
Let $W= \bigcap (\alpha_{B(p_0,x)})_x \cap \bigcap_{i=1}^p (A_0 \setminus U_{2n_i-1}(A_i)) \cap W_{B(p_0,x),y}$. Then each element in this (finite) intersection is as required for $W$ to be in $\mathcal{W}_{p(x,T)}(y)=\mathcal{W}_{p_0,M,\{B(p_0,x)\}}(y)$. Each element in the intersection contains $x$, so $x$ is in $W$.

It remains to show that $W$ is contained in $V$. Clearly $W=V_0 \cap V_1 \cap W_{B(p_0,x),y}$. So it suffices to see that $W_{B(p_0,x),y} \subseteq V$. But for $i>p$, as $y \notin A_i$, by definition we see that $W_{B(p_0,x),y} \subseteq X \setminus A_i$ and $X \setminus A_i \subseteq X \setminus U_{2n_i-1}(A_i)$. Hence $W_{B(p_0,x),y} \subseteq \bigcap_{i=p+1}^k (X \setminus U_{2n_i-1}(A_i))$.

\medskip

\noindent \emph{For (2):} Fix $(p,m,\mathcal{B}_0)$ in $P \times \N \times Q$, We need to show that there is a neighborhood $T$ of it such that $\mathcal{W}_T(x)$ is finite. Indeed we claim $\mathcal{W}_{B(p,x) \times \{(m,\mathcal{B}_0)\}}(x)$ is finite.

Take any $p'$ in $B(p,x)$. Elements of $\mathcal{W}_{p',m,\mathcal{B}_0}(x)$ are intersections of three types. Those of type (iii) are exactly the same as in $\mathcal{W}_{p,m,\mathcal{B}_0}(x)$. Those of type (i) and (ii) come from $A$'s in $(\alpha_{B(p',x)})_x$.
By definition, $B(p',x)$  is either $B(p,m)$ or occurs before it in the enumeration of $\mathcal{B}$. Hence $\{B(p',x) : p' \in B(p,x)\}$ is finite.
So there are only finitely many $A$'s going into types (i) and (ii), and hence $\mathcal{W}_{B(p,x) \times \{(m,\mathcal{B}_0)\}}(x)$ is finite as claimed.
\end{proof}

\subsection*{Further Characterizations}

\begin{theorem}\label{main2b}
Let $X$ be compact. Let $P$ be a topological directed set which is separable metrizable, but not compact, down sets are compact and every convergent sequence has an upper bound.

Then the following are equivalent:

(1) $X$ has a $P$-point finite almost subbase, (2) $X$ has a $P$-point additively Noetherian base, (3) $X$ has a $P$-point finite expandable network, (4) $X$ has a $P$-finite strong point network, and (5) $X$ has a $P$-finite  point network.
\end{theorem}

\begin{theorem} Let $X$ be compact. Then the following are equivalent:
(0) $X$ is Eberlein compact, (1) $X$ has a $\N$-point finite almost subbase,
(2) $X$ has a $\N$-point finite expandable network, and (3) $X$ has a $\N$-finite point network.
\end{theorem}
As mentioned above, the equivalence of (0) and (2) is due to \cite{DJP}.

\begin{theorem}\label{ch_T2} Let $X$ be compact. Then the following are equivalent:
(0) $X$ is Talagrand compact, (1) $X$ has a $\N^\N$-point finite almost subbase,
(2) $X$ has a $\N^\N$-point finite expandable network, and (3) $X$ has a $\N^\N$-finite point network.
\end{theorem}

\begin{theorem}\label{ch_G2} Let $X$ be compact. Then the following are equivalent:
(0) $X$ is Gulko compact, (1) $X$ has a $\K(M)$-point finite almost subbase,
(2) $X$ has a $\K(M)$-point finite expandable network, and (3) $X$ has a $\K(M)$-finite point network, where in (1), (2) and (3) $M$ is some separable metrizable space.
\end{theorem}

At first glance the equivalence  of (0) and (2) in the preceding two results correspond to results of
Garcia, Oncina and Orihuela \cite{GOO}. However Garcia et al used a variant of expandable networks that we shall call \emph{expandable networks$^*$} and established that a compact space is Gulko (respectively, Talagrand) if and only it has a $P$-point finite expandable network$^*$ for $P=$ some $\K(M)$ (respectively, $P=\N^\N$).

Let $\mathcal{N}$ be a family of subsets of a space $X$. Then $\mathcal{N}$ is a $P$-point finite expandable network$^*$ if  it can be indexed $\mathcal{N}=\{N_i : i \in \mathcal{I}\}$ where $\mathcal{I}=\bigcup \{\mathcal{I}_p : p\in P\}$
and for every $i\in \mathcal{I}$ there exists an open set $V_i\supseteq N_i$ in $X$ such that:

(a) if $p \le p'$ then $\{V_i : i \in \mathcal{I}_p\} \subseteq \{V_i : i \in \mathcal{I}_{p'}\}$,

(b) for any $x\in X$ and $p\in P$, the set $\{i: i \in \mathcal{I}_p$ and $x\in V_i\}$ is finite, and

(c) for any open set $U$ containing a point $x$ of $X$, there is an $N$ in $\mathcal{N}$ such that $x \in N \subseteq U$.

In the opinion of the authors our definitions of `expandable network' and `$P$-point-finite' are the most natural ones. It is certainly what we need for the results of this paper, and in particular there is no natural connexion between an  expandable network$^*$ and point networks. Of course Theorems~\ref{ch_T2}, and~\ref{ch_G2} say that the two concepts coincide where the authors of \cite{GOO} used them.

\subsection*{Without Compactness}

Theorem~\ref{main2b} is stated for compact spaces. However the implications `(1) $\implies$ (2)', `(3) $\iff$ (4)', and `(4) $\implies$ (5)' all hold for general spaces. The authors do not know of an example when condition (5) does not give condition (4).

The following example gives a strong (non-compact) counter-example to `(3) $\implies$ (2)' with $P=\N$.

\begin{lemma}
There is a countable space with a  point finite expandable network but no $\N$-point additively Noetherian base.
\end{lemma}
\begin{proof}
Let $\Fan$ be the Frechet fan. So $\Fan$ is obtained from $\omega \times (\omega +1)$ by identifying the end points $(m,\omega)$ to a point $\ast$. Let $\mathcal{N}= \{ (\{(m,n)\},\{(m,n)\}) : m,n \in \omega\} \cup \{(\{\ast\} , \Fan)\}$. This can easily be checked to be a point finite expandable network.

To complete the proof we show that $\ast$ has no $\N$-additively Noetherian base. If $U$ is an open neighborhood of $\ast$ define $f_U$ in $\omega^\omega$ by $f_U(m) = \min \{n : (m,n) \in U\}$. Conversely, if $g \in \omega^\omega$ then define $U_g=\{(m,n) : g(m) \le n\} \cup \{\ast\}$ an open neighborhood of $\ast$.

Suppose $\bigcup \{\mathcal{B}_n : n \in \N\}$ is a neighborhood base for $\ast$. We show there is an $N$ such that $\mathcal{B}_N$ is not additively Noetherian. Observe that $\bigcup \{\mathcal{F}_n : n \in \N\}$ where $\mathcal{F}_n=\{ f_U : U \in \mathcal{B}_n\}$ is a dominating family in $(\omega^\omega,\le^*)$ (here $\le^*$ is the mod-finite order). So for some $N$ we have that $\mathcal{F}_N$ is dominating in $(\omega^\omega,\le^*)$. Let $\mathcal{B}=\mathcal{B}_N$ and $\mathcal{F}=\mathcal{F}_N$. Hence there is a $\mathcal{G}=\{ g_\alpha : \alpha < \mathfrak{b}\}$ a subset of $\mathcal{F}$ such that $\mathcal{G}$ is unbounded and $g_\alpha <^* g_\beta$ if $\alpha < \beta$, where $\mathfrak{b}$ is the minimal size of an unbounded family in $(\omega^\omega,\le^*)$.

Define, by recursion, a decreasing sequence $(S_n)_{n\in \omega}$ of cofinal subsets of $\mathfrak{b}$ such that, for all $\alpha \in S_n$, $g_\alpha (n)=k_n$, for some fixed $k_n$. Define $k$ in $\omega^\omega$ by $k(n)=k_n$.

Set $n_0=0$. As $\mathcal{G}$ is unbounded, and $S_0$ is cofinal in $\mathfrak{b}$, there is an $\alpha_0$ in $S_0$ such that $g_{\alpha_0} \not\le^* k$. So, for infinitely many $n$, $k(n) < g_{\alpha_0}(n)$. Let $n_1$ be the first such $n$. Next choose $\alpha_1 \in S_{n_1}$, $\alpha_0 < \alpha_1$. Then $g_{\alpha_0} <^* g_{\alpha_1}$ and $g_{\alpha_1}(n_1)=k(n_1) < g_{\alpha_0}(n_1)$. And now choose $n_2 >n_1$ such that $k(n_2) < g_{\alpha_0}(n_2) < g_{\alpha_1}(n_2)$.

Proceeding inductively, we find a strictly increasing sequence of ordinals $(\alpha_i)_{i \in \omega}$ and integers $(n_i)_i$ such that, for every $p \in \omega$, $g_{\alpha_p}(n_p)=k(n_p)< \min \{g_{\alpha_i}(n_p) : 0 \le i < p\}$ and $(g_{\alpha_i})_i$ is strictly increasing with respect to $<^*$.

The $g_{\alpha_i}$'s are in $\mathcal{F}$, and hence of the form $g_{\alpha_i}=f_{U_i}$ for some $U_i$ in $\mathcal{B}$. So, by definition of $f_{U_i}$, for each $p$, $(n_p,g_{\alpha_p}(n_p)) \in U_p$, but, since $g_{\alpha_p}(n_p) < \min \{g_{\alpha_i}(n_p) : 0 \le i < p\}$, for $i<p$, we have $(n_p,g_{\alpha_p}(n_p)) \notin U_i$.

From the last line and Lemma~\ref{Noeth} (4), it follows that $\mathcal{B}=\mathcal{B}_N$ is not additively Noetherian, as desired.
\end{proof}

Our last example shows that for non-compact spaces condition (2) does not imply (5), even with $P=1$ in the hypothesis and `$P$ calibre $(\omega_1,\omega)$' in the conclusion.
\begin{lemma}
There is a space with point additively Noetherian base but no countable point network, and hence no $P$-finite point network for any $P$ with calibre $(\omega_1,\omega)$.
\end{lemma}

\begin{proof} Let $X=\{0,1\}^{\omega_1}$. For any $\vx\in X$, a basic open neighborhood of $\vx$ is $B_\alpha(\vx)=\{\vy: \vy(\gamma)=\vx(\gamma) \text{ for all }\gamma<\alpha\}$ for $\alpha<\omega_1$. Then $\mathcal{B}=\{B_\alpha(\vx):\vx\in X \text{ and }\alpha<\omega_1\}$ is a base. Note that: (1) given any two elements of $\mathcal{B}$, either they have empty intersection or one is a subset of the other; (2) any countable intersection of $\mathcal{B}$ is either empty or in $\mathcal{B}$; and (3) any  countable subset of $X$ is closed and discrete. For any $\vx, \vy\in X$, $\vx\in B_\alpha(\vy)$ implies that $B_\alpha(\vx)=B_\alpha(\vy)$. Therefore, for any $\vx\in X$, $(\mathcal{B})_x=\{B_\alpha(\vx) \}$. We can see that, $(\mathcal{B})_x$ ordered by reverse inclusion is order-isomorphic to $\omega_1$. So $\mathcal{B}$ is point additively Noetherian.

Suppose, for a contradiction, that $X$ has a countable point network, $\mathcal{W}=\bigcup\{\mathcal{W}(\vx): \vx\in X\}$. Also, for each $\vx\in X$ and an open set $U$ containing $\vx$, we have the corresponding $V(\vx, U)$. We can assume that $V(\vx, U)$ is in $\mathcal{B}$.

Pick $\vx_0\in X$, and let $V_0=V(\vx_0, X)$. If $\alpha$ is a successor, pick $\vx_\alpha\in V_{\alpha^-}$ and  $\vx_\alpha\neq \vx_\beta$ for all $\beta<\alpha$, and also let $V_\alpha$ be a basic open neighborhood of $\vx_\alpha$ which is a subset of $V(x_\alpha, V_{\alpha^-})\cap B_\alpha(\vx_\alpha)\setminus \{\vx_\beta: \beta<\alpha\}$. If $\alpha$ is a limit ordinal, pick $\vx_\alpha\in\bigcap \{V_{\beta}: \beta<\alpha\}$ also $\vx_\alpha\neq \vx_\beta$ for all $\beta<\alpha$, and let $V_\alpha$ be a basic open neighborhood of $\vx_\alpha$ which is a subset of $V(\vx_\alpha, \bigcap \{V_{\beta}: \beta<\alpha\})\cap B_\alpha(\vx_\alpha)\setminus \{\vx_\beta: \beta<\alpha\}$. We get a transfinite sequence of points $\{\vx_\alpha: \alpha<\omega_1\}$, together with a decreasing transfinite sequence of basic open sets  $\{V_\alpha: \alpha<\omega_1\}$. Note that $\bigcap \{V_\alpha: \alpha <\omega_1\}$ is not empty. Then there is a $\vy\in \bigcap \{V_\alpha: \alpha <\omega_1\}$. For each $\alpha<\omega_1$, there is a $W_\alpha\in\mathcal{W}(\vy)$ such that $\vx_{\alpha+1}\in W_\alpha\subseteq V_{\alpha}$. Since $\mathcal{W}(\vy)$ is countable, there exists an uncountable subset $A$ of $\omega_1$ such that $W_\alpha$ is a fixed $W$ for each $\alpha\in A$. Since $\vy\in B_\alpha(\vx_\alpha)$, $B_\alpha(\vx_\alpha)=B_\alpha(\vy)$. Choose $\beta\in A$, then $\vx_{\beta+1}\in W$. Since $\vx_{\beta+1}\neq \vy$ and $A$ is uncountable, we can find $\gamma\in A$ such that $\vx_{\beta+1} \notin B_{\gamma}(\vy)$. However, since $\vy\in V(\vx_{\gamma+1}, V_\gamma)$, we have $W\subseteq V_\gamma$ which is a subset of $B_{\gamma}(\vy)$ which contradicts the fact $\vx_{\beta+1} \notin B_{\gamma}(\vy)$. This complete the proof.
\end{proof}


\begin{thebibliography}{9}

\bibitem{AL} D. Amir and J. Lindenstrauss, The structure of weakly compact sets in Banach spaces, Ann. of Math., 88 (1968), 34-46.


\bibitem{BRW77} Y. Benyamini, M. E. Rudin, and M. Wage, Continuous images of weakly compact subsets of Banach spaces, Pacific J. Math. 70 (1977), no. 2, 309–324.

\bibitem{CR} P.J. Collins and A.W. Roscoe, Criteria for metrizability, Proc. Amer. Math. Soc. 90(1984), 631-640.

\bibitem{Dim} G. Dimov, On Eberlein spaces and related spaces, C. R. Acad. Sc. Paris, 304, S'erie
I, no. 9 (1987), 233-235.

\bibitem{Dim2} G. Dimov, Baire subspaces of $c_0(\Gamma)$ have dense $G_{\delta}$ metrizable subsets, Rend. Circ. Mat. Palermo, Ser. II, Suppl. 18 (1988)
275–285


\bibitem{DJP} A. Dow, H. Junnila, J. Pelant, Weak covering properties of weak topologies, Proc. London
Math. Soc. (3) 75 (1997), no. 2, 349–368.

\bibitem{Eng} R. Engelking, General topology, PWN, Warsaw, 1977.

\bibitem{GOO} F. García, L. Oncina and J. Orihuela, Network characterization of Gul'ko compact spaces and their relatives, J. Math. Anal. Appl. 297 (2004), no. 2, 791-811.

\bibitem{Gruen} G. Gruenhage, Covering properties on $X^2 \setminus \Delta$, $W$-sets, and compact subsets of $\Sigma$-products, Topology Appl. 17 (1984), 287–304.

\bibitem{Jun} H. Junnila,
Eberlein compact spaces and continuous semilattices. General topology and its relations to modern analysis and algebra, VI (Prague, 1986), 297-322,
Res. Exp. Math., 16, Heldermann, Berlin, 1988.


\bibitem{Merc} S. Mercourakis, On weakly countably determined Banach spaces, Trans. Amer. Math. Soc. 300 (1987), 307-327


\bibitem{Ny} P. Nyikos,
On the product of metacompact spaces. I. Connections with hereditary compactness.
Amer. J. Math. 100 (1978), no. 4, 829-835.

\bibitem{Ro} H.P. Rosenthal, The heredity problem for weakly compactly generated Banach spaces, Compos. Math. 28: 1 (1974), pp. 88-111.

\bibitem{Sokol} G. A. Sokolov, On some classes of compact spaces lying in $\Sigma$-products, Comm. Math. Univ. Carolin. 25 (1984), 219-231.

\end{thebibliography}
\end{document}